\documentclass[10pt,reqno]{amsart}
\usepackage{graphicx,epsfig}
\usepackage{epstopdf}
\usepackage{pdfpages}
\usepackage{amssymb}
\usepackage{empheq}
\usepackage{cases}
\usepackage{amsthm,amsmath}
\usepackage{caption,lipsum}
\usepackage{stmaryrd}
\usepackage{tabularx}
\usepackage{color}
\usepackage{empheq,float}
\DeclareGraphicsExtensions{.eps}
\usepackage{hyperref}

\newtheorem{theorem}{Theorem}[section]
\newtheorem{lemma}[theorem]{Lemma}

\theoremstyle{definition}

\theoremstyle{remark}
\newtheorem{remark}[theorem]{Remark}

\def\R{\mathbb{R}}
\def\Z{\mathbb{Z}}
\def\T{\mathbb{T}}

\newcommand{\fe}{\mathrm{e}}

\newcommand{\bR}{{\mathbb R}}

\newcommand{\bT}{{\mathbb T}}
\newcommand{\bN}{{\mathbb N}}
\newcommand{\bZ}{{\mathbb Z}}

\numberwithin{equation}{section}

\begin{document}

\title[Low-regularity integrator for KdV equation]{Optimal convergence of a second order low-regularity integrator for the KdV equation}

\author[Y. Wu]{Yifei Wu}
\address{\hspace*{-12pt}Y.~Wu: Center for Applied Mathematics, Tianjin University, 300072, Tianjin, China}
\email{yerfmath@gmail.com}

\author[X. Zhao]{Xiaofei Zhao}
\address{\hspace*{-12pt}X.~Zhao: School of Mathematics and Statistics \& Computational Sciences Hubei Key Laboratory, Wuhan University, Wuhan, 430072, China}
\email{matzhxf@whu.edu.cn}

%

\date{}

\dedicatory{}

\begin{abstract}
In this paper, we establish the optimal convergence result of a second order exponential-type integrator from (136, Numer. Math., 2017) for solving the KdV equation under rough initial data. The scheme is explicit and efficient to implement. By rigorous error analysis, we show that the scheme provides the second order accuracy in $H^\gamma$ for initial data in $H^{\gamma+4}$ for any $\gamma\geq0$, where the regularity requirement is lower than the classical methods. The result is confirmed by numerical experiments and comparisons are made with the Strang splitting scheme.
\end{abstract}

\maketitle

\section{Introduction}
During the past few years, to solve efficiently problems with rough initial data, the so-called low-regularity numerical integrators have been proposed for some dispersive models. Compared to the classical numerical discretizations such as the finite difference methods or standard exponential integrators \cite{Hochbruck} or splitting schemes \cite{Splitting}, the low-regularity methods require less regularity of the solution to reach their optimal convergence rates. For example, for the cubic nonlinear Schr\"{o}dinger equation, the first order convergence in $H^r$ has been achieved under only $H^{r+1}$-data \cite{lownls2,lownls3}, and for the  one-dimensional quadratic nonlinear Schr\"{o}dinger equation \cite{lownls2,lownls3} or the nonlinear Dirac equations \cite{diraclow}, only $H^{r}$-data is needed. In this work, we are concerned with the Korteweg-de Vries (KdV)  equation, which is a classical mathematical model for the waves on shallow water surfaces, under the rough initial data on a torus:
\begin{equation}\label{model}
 \left\{\begin{split}
& \partial_tu(t,x)+\partial_x^3u(t,x)
 =\frac{1}{2}\partial_x(u(t,x))^2,
 \quad t>0,\ x\in\bT,\\
 &u(0,x)=u_0(x),\quad x\in\bT,
 \end{split}\right.
\end{equation}
where $\bT=(0,2\pi)$, $u=u(t,x):\bR^{+}\times\bT\to\bR$ is the unknown and $u_0\in H^{s_0}(\bT)$ with some $0\leq s_0<\infty$ is a given initial data. For the numerical studies of the KdV equation (\ref{model}) under smooth enough initial data cases, we refer to \cite{Spectralkdv0,splittingJCP,FD1,DG1,Spectralkdv1,Fourierkdv,Spectralkdv2,splitting0,DG2}.
However, in practice the initial data may not be ideally smooth due to multiply reasons such as  measurements or noise \cite{kdv-wellposed}. Analytically, the global well-posedness of (\ref{model}) on the torus under rough data has already been established in \cite{CKSTT-03-KDV, KaTo}. That is, for any $u_0\in H^{s_0}(\T), s_0\ge -\frac{1}{2}$ and any positive time $T >0$, there exists a unique solution of \eqref{model} in a certain Banach space of functions
$X \subset C([0, T ];H^{s_0}(\T))$ as established in \cite{CKSTT-03-KDV} by PDE methods, and \cite{KaTo} extended the result to $s_0\ge -1$ by the inverse scattering method. 
See also \cite{BaoWu,CKSTT-04-KDV} for the theoretical studies of the generalized KdV equations.

The numerical aspects of the KdV equation (\ref{model}) under rough data have been addressed in \cite{FD-kdv,kdv-wellposed,kdv-kath,splitting2}.
By introducing the twisted variable $v:=\fe^{\partial_x^3t}u$ and the Duhamel's formula at $t_n=n\tau$ with $\tau>0$ the time step:
\begin{equation}\label{v eq}
  v(t_n+\tau,x)=v(t_n,x)+\frac{1}{2}\int_0^\tau\fe^{(t_n+s)\partial_x^3}
  \partial_x\left(\fe^{-(t_n+s)\partial_x^3}v(t_n+s,x)\right)^2ds,
\end{equation}
 \cite{kdv-kath} proposed an exponential-type numerical scheme by letting $v(t_n+s,x)\approx v(t_n,x)$, and then the integration for $s$ was found exactly and explicitly in the physical space. Comparing to classical first order methods, this strategy gives rise to a first order numerical scheme that allows the low regularity requirement for solving the KdV equation (\ref{model}):
$$\|u(t_n,\cdot)-u^n\|_{H^1}\lesssim \tau,$$
up to some finite time for $u\in H^{3}$ as was proved in \cite{kdv-kath}, where $u^n$ denotes the numerical solution at $t_n$. That is to say for solving the KdV, in order to reach the first order accuracy, the regularity requirement of such exponential-type low-regularity integrator is the boundedness of two additional spatial derivatives of the solution. This requirement is to some extend essential for the direct integration methods due to the Burgers-type nonlinearity in the KdV equation \cite{Schratznew}.

To push the convergence rate of such low-regularity integrators to the second order, one natural way is to use the idea of Picard iteration, i.e. using the first order scheme to approximate the solution in the nonlinearity in the Duhamel's formula (\ref{v eq}) and then integrate exactly for $s$ in the Fourier space. Obviously by doing this, the Burgers-nonlinearity will add one more loss of derivative. What makes it worse after the iteration is that, the quadratic nonlinearity in the KdV equation will generate four different modes in the Fourier space, which after the time integration can not be translated back to an explicit function in the physical space, and the scheme will then have to deal with a heavy convolution in the Fourier space. A practical strategy as a compensation for computational efficiency is to drop some Fourier modes before the time integration, in order to get back to the physical space, which consequently further asks for more regularity of the solution for obtaining the second order convergence rate. See the very recent effort on the nonlinear Schr\"{o}dinger equation for designing such second order low-regularity integrators in \cite{lownls2}. Together with the loss from the Burgers-nonlinearity, one major concern of the Picard iteration approach is that the second order convergence rate is achieved in $H^\gamma$-norm for solutions in as least $H^{\gamma+4}$ space. See a precise construction in such approach for the KdV equation in the recent work \cite{Schratznew}.

In this work, instead of the Picard iteration, we consider the construction for a second order low-regularity integrator (LRI) for the KdV (\ref{model}) by simply the Taylor's expansion in (\ref{v eq}) as
$$v(t_n+s,x)\approx v(t_n,x)+s\partial_tv(t_n,x).$$
This strategy has been outlined in \cite{kdv-kath}, and it leads to an explicit and efficient scheme.
We are aiming to establish the optimal convergence result of this second order LRI scheme by showing that
$$\|u(t_n,\cdot)-u^n\|_{H^\gamma}\lesssim \tau^2,$$
up to some fixed time for $u\in H^{\gamma+4}$ for any $\gamma\geq0$. The proof relies on a key fact that
$$
\partial_t^2v=\frac32\fe^{t\partial_x^3}\partial_{x}^2
\left(\fe^{-t\partial_x^3}\partial_xv\right)^2
+\frac13\fe^{t\partial_x^3}\partial_{x}^2
\left(\fe^{-t\partial_x^3}v\right)^3,
$$
 and some tools from the harmonic analysis  to overcome the absence of algebraic property of $H^\gamma$ when $0\leq\gamma\leq\frac12$.
As shown by our theoretical estimate,  this simplified strategy (compared with Picard iteration) is able to get the desired second order accuracy with only \emph{four} additional bounded spatial derivatives, which is better than what was conjectured in \cite{kdv-kath}, and it matches with the best result that one could expect from the direct Picard iteration as we have explained above and from the recent investigation in \cite{Schratznew}. We shall show by numerical results that the regularity requirement for this second order LRI is sharp.
Compared with other classical second order numerical methods in the literature, for example the Strang splitting method (with exact solutions at the Burgers' step) \cite{splitting1,splitting2} that needs $u\in H^{\gamma+5}$ for the second order convergence rate in $H^\gamma$ for $\gamma\geq1$, the presented LRI saves one spatial derivative and further reduces the requirement of the regularity. Moreover, one should note that the exact solutions at the Burgers' step in the Strang splitting scheme require a nonlinear solver which makes the practical scheme costly.
Hence, the presented LRI is particularly more efficient for solving the KdV (\ref{model}) under rough data, which will be illustrated through numerical experiments in the end.

The rest of the paper is organized as follows. In Section \ref{sec:scheme}, we present the detailed scheme of LRI from \cite{kdv-kath} and give its main convergence theorem. In Section \ref{sec:proof}, we give the proof of the convergence result. Numerical confirmations are reported in Section \ref{sec:numerical} and conclusions are drawn in Section \ref{sec:conclusion}.

\section{Numerical scheme}\label{sec:scheme}
In this section, we present the detailed numerical scheme of the second order exponential-type integrator as outlined in \cite{kdv-kath} and give
its main convergence theorem. In the following for simplicity, we shall assume that the zero-mode/average of the initial value of the (\ref{model}) is zero. Otherwise, we may consider
$$
\tilde u:= u\left(t,x-\widehat{(u_0)}_0 t\right)-\widehat{(u_0)}_0
$$
instead, and one can note that $\tilde u$ also obeys the same KdV equation of \eqref{model} with initial data $\tilde u_0:=u_0-\widehat{(u_0)}_0$. Here we denote $\widehat{(u_0)}_l$ for $l\in\bZ$ as the Fourier coefficients of $u_0$.

With the twisted variable
\begin{equation}\label{twist}v(t,x):=\fe^{t\partial_x^3}u(t,x),\quad t\geq0,\ x\in\bT,\end{equation}
the KdV equation (\ref{model}) becomes:
\begin{equation}\label{vt eq}
  \partial_t v(t,x)=\frac{1}{2}\fe^{t\partial_x^3}
  \partial_x\left(\fe^{-t\partial_x^3}v(t,x)\right)^2,\quad
  t\geq0,\ x\in\bT.
\end{equation}
For $n\geq0$ and $t=t_n+s$, plugging $v(t_n+s,x)\approx v(t_n,x)+s\partial_tv(t_n,x)$ into the Duhamel's formula (\ref{v eq}), we get
\begin{align}
 v(t_{n+1},x)\approx& v(t_n,x)+\frac{1}{2}\int_0^\tau
 \fe^{(t_n+s)\partial_x^3}\partial_x\left[\fe^{-(t_n+s)\partial_x^3}
 \big(v(t_n,x)+s\partial_tv(t_n,x)\big)\right]^2ds\nonumber\\
 \approx&v(t_n,x)+\frac{1}{2}I_1(t_n,x)+I_2(t_n,x),\label{v app}
\end{align}
where
\begin{subequations}\label{I_j def}
\begin{align}
 &I_1(t_n,x):=\int_0^\tau\fe^{(t_n+s)\partial_x^3}
 \partial_x\left(\fe^{-(t_n+s)\partial_x^3}v(t_n,x)\right)^2ds,\\
 & I_2(t_n,x):=\int_0^\tau\fe^{(t_n+s)\partial_x^3}
 \partial_x\left[
 s\left(\fe^{-(t_n+s)\partial_x^3}v(t_n,x)\right)
 \left(\fe^{-(t_n+s)\partial_x^3}\partial_tv(t_n,x)\right)\right]ds.
\end{align}
\end{subequations}
By calculation in the Fourier frequency space and noting the key relation
$$(k_1+k_2)^3-k_1^3-k_2^3=3(k_1+k_2)k_1k_2,$$
the terms in $I_1$ and $I_2$ can be exactly integrated in the physical space. We refer the readers to \cite{kdv-kath} for those detailed calculations. The zero Fourier mode of $I_1$ or $I_2$ is clearly zero, and for the other modes, here we directly write down their explicitly formulas in the physical space:
$$I_1(t_n,x)=\sum_{l\neq0}\widehat{(J_{1})}_l(t_n)\fe^{ilx},\qquad
I_2(t_n,x)=\sum_{l\neq0}\widehat{(J_{2})}_l(t_n)\fe^{ilx},$$
where $\widehat{(J_{1})}_l$ and $\widehat{(J_{2})}_l$ denote  respectively  the Fourier coefficients of functions $J_1$ and $J_2$
\begin{align*}
 J_1(t_n,x):=\frac{1}{3}\fe^{t_{n+1}\partial_x^3}\left(
 \fe^{-t_{n+1}\partial_x^3}\partial_x^{-1}v(t_n,x)\right)^2
 -\frac{1}{3}\fe^{t_n\partial_x^3}\left(\fe^{-t_n\partial_x^3}
 \partial_x^{-1}v(t_n,x)\right)^2,
\end{align*}
and
\begin{align*}
 J_2(t_n,x):=&\frac{\tau}{3}\fe^{t_{n+1}\partial_x^3}\left(
 \fe^{-t_{n+1}\partial_x^3}\partial_x^{-1}v(t_n,x)\right)
 \left(\fe^{-t_{n+1}\partial_x^3}\partial_x^{-1}\partial_tv(t_n,x)\right)\\
 &-\frac{1}{9}\fe^{t_{n+1}\partial_x^3}\partial_x^{-1}
 \left(\fe^{-t_{n+1}\partial_x^3}\partial_x^{-2}v(t_n,x)\right)
 \left(\fe^{-t_{n+1}\partial_x^3}
 \partial_x^{-2}\partial_tv(t_n,x)\right)\\
 &+\frac{1}{9}\fe^{t_{n}\partial_x^3}\partial_x^{-1}
 \left(\fe^{-t_n\partial_x^3}\partial_x^{-2}v(t_n,x)\right)
 \left(\fe^{-t_n\partial_x^3}
 \partial_x^{-2}\partial_tv(t_n,x)\right).
\end{align*}
Here the operator $\partial_x^{-m}\ (m\in\bN)$ is defined for function $f(x)
\in L^2(\bT)$ as
$$\partial_x^{-m}f(x):=\sum_{l\neq0}(il)^{-m}
\widehat{f}_l\fe^{ilx},\quad x\in \bT.$$
Noting (\ref{vt eq}), we substitute
$$\partial_tv(t_n,x)=\frac{1}{2}\fe^{t_n\partial_x^3}
  \partial_x\left(\fe^{-t_n\partial_x^3}v(t_n,x)\right)^2$$
 into $J_2(t_n,x)$, then the approximation (\ref{v app}) forms an update from $t_n$ to $t_{n+1}$ for $v(t,x)$, and by reverting the change of variable (\ref{twist}) we get the scheme for $u(t,x)$.

The detailed \emph{second order low-regularity integrator (LRI)} for solving the KdV equation (\ref{model}) then reads: denote
$u^n=u^n(x)\approx u(t_n,x)$ as the numerical solution for $n\geq0$, $I_j^n=I_j^n(x)\approx \fe^{-t_{n+1}\partial_x^3}I_j(t_n,x)$ for $j=1,2$ as approximations of (\ref{I_j def}) and let $u^0=u_0(x)$, then
\begin{align}\label{LRI}
u^{n+1}=\fe^{-\tau\partial_x^3}u^n+\frac{1}{2}I_1^n
+I_2^n,\quad n=0,1,\ldots,
\end{align}
where
$$I_1^n=\sum_{l\neq0}\widehat{(J_{1}^n)}_l\fe^{ilx},\quad
I_2^n=\sum_{l\neq0}\widehat{(J_{2}^n)}_l\fe^{ilx},$$
and
\begin{align*}
 J_1^n:=&\frac{1}{3}\left(
 \fe^{-\tau\partial_x^3}\partial_x^{-1}u^n\right)^2
 -\frac{1}{3}\fe^{-\tau\partial_x^3}\left(
 \partial_x^{-1}u^n\right)^2,\\
 J_2^n:=&\frac{\tau}{6}\left(
 \fe^{-\tau\partial_x^3}\partial_x^{-1}u^n\right)
 \left(\fe^{-\tau\partial_x^3}(u^n)^2-\widehat{((u^n)^2)}_0
 \right)-\frac{1}{18}\partial_x^{-1}
 \left(\fe^{-\tau\partial_x^3}\partial_x^{-2}u^n\right)
 \left(\fe^{-\tau\partial_x^3}
 \partial_x^{-1}(u^n)^2\right)\\
 &+\frac{1}{18}\fe^{-\tau\partial_x^3}\partial_x^{-1}
 \left(\partial_x^{-2}u^n\right)
 \left( \partial_x^{-1}(u^n)^2\right).
\end{align*}

The above LRI (\ref{LRI}) is explicit and preserves the mass of the KdV equation (\ref{model}) at the discrete level, i.e.
$$\int_\bT u^n(x)dx\equiv\int_\bT u_0(x)dx,\quad n=0,1,\ldots. $$
In practice, the Fourier series in the scheme are truncated to an integer $N>0$ and the computational of the Fourier coefficients are obtained by the trigonometric quadrature \cite{Shen}. Then the fully discretized LRI scheme is efficient thanks to the fast Fourier transform with computational cost $O(N\log N)$ at each time level, and it has no CFL-type conditions.

Now, we state the convergence theorem of the presented (semi-discretized) LRI method (\ref{LRI}) as the \textbf{main result} of the paper and the proof is given in the next section.

\begin{theorem}\label{thm:convergence}
Let $u^n$ be the numerical solution of the KdV (\ref{model}) obtained from the LRI scheme (\ref{LRI}) up to some fixed time $T>0$. Under assumption that $u_0\in H^{\gamma+4}(\bT)$ for some $\gamma\geq0$,  there exists constants $\tau_0,C>0$ such that for  any $0<\tau\leq \tau_0$, we have
\begin{equation}\label{main result}
  \|u(t_n,\cdot)-u^n\|_{H^\gamma}\leq C\tau^2,\quad n=0,1\ldots,\frac{T}{\tau},
\end{equation}
 where the constants $\tau_0$ and $C$ depend only on $T$ and $\|u\|_{L^\infty((0,T);H^{\gamma+4})}$.
\end{theorem}

We shall show later by numerical results that the estimate and the regularity assumption in the above convergence theorem for LRI are sharp. The convergence result of LRI is indeed better than conjectured in \cite{kdv-kath}.

\section{Convergence analysis}\label{sec:proof}

In this section, we give the rigorous proof of the main result. To do so, we shall firstly introduce some tools from harmonic analysis in subsection \ref{subsec1}, and then establish the stability result and the local error estimate,  respectively in subsection \ref{subsec3} and subsection \ref{subsec2}. The final proof of  Theorem \ref{thm:convergence} is given in subsection \ref{subsec4}.

\subsection{Some notations and tools}\label{subsec1}
For convenience, we introduce some notations and definitions, some of which are employed from \cite{CKSTT-03-KDV}. We use $A\lesssim B$ or $B\gtrsim A$ to denote the statement that $A\leq CB$ for some absolute constant $C>0$ which may
vary from line to line but independent of $\tau$ or $n$, and we denote $A\sim B$ for
$A\lesssim B\lesssim A$.
We define
$(d\xi)$ to be  the normalized counting measure on
$\Z$ such that
\begin{equation*}
\displaystyle\int
a(\xi)\,(d\xi)
=
\sum\limits_{\xi\in
\Z} a(\xi).
\end{equation*}
The Fourier transform of a function $f$ on $\T$ is defined by
$$
\mathcal{F}(f)(\xi)=\widehat{f}(\xi)=\displaystyle\frac{1}{2\pi}\int_\bT \fe^{- i  x\xi}f(x)\,dx,
$$
and thus the Fourier inversion formula
$$
f(x)=\displaystyle\int \fe^{ i  x\xi} \widehat{f}(\xi)\,(d\xi).
$$
Then the following usual properties of the Fourier transform hold:
\begin{eqnarray*}
 &\|f\|_{L^2(\bT)}
 = \big\|\widehat{f}\big\|_{L^2((d \xi))} \quad \mbox{(Plancherel)}; \\
 &\displaystyle\langle f,g\rangle=\int_\bT f(x)\overline{g(x)}\,dx
 = \displaystyle\int \widehat{f}(\xi)\overline{\widehat{g}(\xi)}\,(d\xi)\quad \mbox{(Parseval)} ; \\
 & \widehat{(fg)}(\xi)=\displaystyle\int
  \widehat{f}(\xi-\xi_1)\widehat{g}(\xi_1) \,(d\xi_1) \quad \mbox{(Convolution)}.
\end{eqnarray*}
The Sobolev space $H^s(\bT)$ for $s\geq0$ has the equivalent norm,
$$
\big\|f\big\|_{H^s(\bT)}=
\big\|J^sf\big\|_{L^2(\bT)}=\left\|\langle\xi\rangle^{s}\widehat{f}(\xi)\right\|_{L^2((d\xi))},
$$
where we denote the operator
$$J^s=(1-\partial_{xx})^\frac s2,\quad
\mbox{and}\quad \langle\cdot\rangle=(1+|\cdot|^2)^{1/2}.$$

As a tool to overcome the absence of the algebraic property of $H^s$ when $s\leq\frac12$, we will frequently call the following Kato-Ponce inequality, where  a general form was proved in \cite{Kato-Ponce} originally and an important progress in the endpoint case was made in \cite{BoLi-KatoPonce, Li-KatoPonce} very recently.
\begin{lemma}\label{lem:kato-Ponce}(Kato-Ponce inequality) The following inequalities hold:
\begin{itemize}
  \item[(i)]
  For any $\gamma\ge 0, \gamma_1>\frac12$, $f,g\in H^{\gamma}\cap  H^{\gamma_1}$, then
\begin{align*}
\|J^\gamma (fg)\|\lesssim \|f\|_{H^\gamma}\|g\|_{H^{\gamma_1}}+ \|f\|_{H^{\gamma_1}}\|g\|_{H^{\gamma}}.
\end{align*}
In particular, if $\gamma>\frac12$, then
\begin{align*}
\|J^\gamma (fg)\|\lesssim \|f\|_{H^\gamma}\|g\|_{H^{\gamma}}.
\end{align*}
  \item[(ii)]
For any  $\gamma\ge 0, \gamma_1>\frac12$, $f\in H^{\gamma+\gamma_1} ,g\in H^{\gamma}$, then
\begin{align*}
\|J^\gamma (fg)\|\lesssim \|f\|_{H^{\gamma+\gamma_1}}\|g\|_{H^{\gamma}}.
\end{align*}
\end{itemize}
\end{lemma}

Moreover, we will need the following specific commutator estimate. Here the commutator is defined as $[A,B]=AB-BA$.
\begin{lemma}\label{lem:kato-Ponce-2}
Let $f,g$ be the Schwartz functions. If $0\le \gamma\le 1$, then the following inequality holds for any $\gamma_1>\frac12$:
\begin{align*}
\big\|[J^\gamma, f]\partial_x g\big\|_{L^2}\le C\|f\|_{H^{1+\gamma_1}}\|g\|_{H^\gamma}.
\end{align*}
Furthermore, if $\gamma>1$, then
\begin{align*}
\big\|[J^\gamma, f]\partial_x g\big\|_{L^2}\le C \Big(\|f\|_{H^{\gamma+\gamma_1}}\|g\|_{H^1}+\|f\|_{H^{1+\gamma_1}}\|g\|_{H^\gamma}\Big),
\end{align*}
or
\begin{align*}
\big\|[J^\gamma, f]\partial_x g\big\|_{L^2}\le C \Big(\|f\|_{H^{\gamma}}\|g\|_{H^{1+\gamma_1}}+\|f\|_{H^{1+\gamma_1}}\|g\|_{H^\gamma}\Big).
\end{align*}
\end{lemma}
\begin{proof}
Taking the Fourier transform on $[J^\gamma, f]\partial_x g$, we get
\begin{align*}
\mathcal F\Big([J^\gamma, f]\partial_x g\Big)(\xi)=i\int_{\xi=\xi_1+\xi_2}\big(\langle \xi\rangle^\gamma-\langle \xi_2\rangle^\gamma\big)\xi_2 \widehat f(\xi_1)\widehat g(\xi_2)\,(d\xi_1).
\end{align*}
We assume that $\widehat f$ and $\widehat g$ are positive, otherwise one may replace them by $|\widehat f|$ and $|\widehat g|$. We could also assume that $\xi_1\ne 0 $ and $ \xi_2\ne 0$, otherwise the term in the above integral vanishes. Denote
$$
\Omega_1=\left\{(\xi,\xi_1,\xi_2):\ \xi=\xi_1+\xi_2,\  |\xi_2|\le \frac1{10}|\xi|\right\},\quad
\Omega_2=\left\{(\xi,\xi_1,\xi_2):\ \xi=\xi_1+\xi_2,\  |\xi_2|> \frac1{10}|\xi|\right\}.
$$
Then by Plancherel's identity,
\begin{align*}
\big\|[J^\gamma, f]\partial_x g\big\|_{L^2}\le \Big\|\int_{\Omega_1}\big(\langle \xi\rangle^\gamma-\langle \xi_2\rangle^\gamma\big)\xi_2 \widehat f(\xi_1)\widehat g(\xi_2)\,(d\xi_1)\Big\|_{L^2}
+\Big\|\int_{\Omega_2}\big(\langle \xi\rangle^\gamma-\langle \xi_2\rangle^\gamma\big)\xi_2 \widehat f(\xi_1)\widehat g(\xi_2)\,(d\xi_1)\Big\|_{L^2}.
\end{align*}

In $\Omega_1$: $ |\xi_2|\le \frac1{10}|\xi|$.  Then, $\frac9{10}|\xi|\le |\xi_1|\le \frac{11}{10}|\xi|$. When $0\le \gamma\le 1$,
$
\big|\langle \xi\rangle^\gamma-\langle \xi_2\rangle^\gamma\big||\xi_2|\lesssim \langle \xi_1\rangle|\xi_2|^\gamma.
$
Hence we have
\begin{align*}
\Big|\int_{\Omega_1}\big(\langle \xi\rangle^\gamma-\langle \xi_2\rangle^\gamma\big)\xi_2 \widehat f(\xi_1)\widehat g(\xi_2)\,(d\xi_1)\Big|\lesssim \int_{\xi=\xi_1+\xi_2}\langle \xi_1\rangle|\xi_2|^\gamma \widehat f(\xi_1)\widehat g(\xi_2)\,(d\xi_1)
=\mathcal F\Big(\langle \nabla\rangle f\cdot|\nabla|^sg\Big)(\xi).
\end{align*}
Hence, by Plancherel's identity and Sobolev's inequality,
\begin{align*}
\Big\|\int_{\Omega_1}\big(\langle \xi\rangle^\gamma-\langle \xi_2\rangle^\gamma\big)\xi_2 \widehat f(\xi_1)\widehat g(\xi_2)\,(d\xi_1)\Big\|_{L^2}
&\le \Big\|\mathcal F\Big(\langle \nabla\rangle f\cdot|\nabla|^sg\Big)(\xi)\Big\|_{L^2}\\
&\lesssim \big\|\langle \nabla\rangle f\cdot|\nabla|^sg\big\|_{L^2}\lesssim \big\|\langle \nabla\rangle f\big\|_{L^\infty}\big\||\nabla|^sg\big\|_{L^2}\\
&\lesssim \big\|\langle \nabla\rangle^{1+\gamma_1} f\big\|_{L^2}\big\||\nabla|^sg\big\|_{L^2}.
\end{align*}
When $\gamma>1$,  we have
$
\big|\langle \xi\rangle^\gamma-\langle \xi_2\rangle^\gamma\big||\xi_2|\lesssim \langle \xi_1\rangle^\gamma|\xi_2|.
$
Then similarly,
\begin{align*}
\Big\|\int_{\Omega_1}\big(\langle \xi\rangle^\gamma-\langle \xi_2\rangle^\gamma\big)\xi_2 \widehat f(\xi_1)\widehat g(\xi_2)\,(d\xi_1)\Big\|_{L^2}
&\lesssim \min\Big\{\big\|\langle \nabla\rangle^{\gamma+\gamma_1} f\big\|_{L^2}\big\|\nabla g\big\|_{L^2}
,\big\|\langle \nabla\rangle^{\gamma} f\big\|_{L^2}\big\|\langle \nabla\rangle^{\gamma_1}\nabla g\big\|_{L^2}\Big\},
\end{align*}
by $L^\infty L^2$-H\"older's or $L^2L^\infty$-H\"older's inequality.
These give the desired result in Case 1.

In $\Omega_2$: $|\xi_2|> \frac1{10}|\xi|$. Then
$
\big|\langle \xi\rangle^\gamma-\langle \xi_2\rangle^\gamma\big||\xi_2|\lesssim\langle\xi_1\rangle |\xi_2|^\gamma.
$
Hence, similar as above, we obtain
\begin{align*}
\Big\|\int_{\Omega_2}\big(\langle \xi\rangle^\gamma-\langle \xi_2\rangle^\gamma\big)\xi_2 \widehat f(\xi_1)\widehat g(\xi_2)\,(d\xi_1)\Big\|_{L^2}
&\lesssim \big\|\langle \nabla\rangle^{1+\gamma_1} f\big\|_{L^2}\big\| |\nabla|^sg\big\|_{L^2}.
\end{align*}
This gives the desired result in $\Omega_2$, and the lemma is proved.
\end{proof}

\begin{remark}
In the following of the section, we shall just adopt a weaker version of the estimates from Lemma \ref{lem:kato-Ponce-2}: For any $\gamma\ge 0$ and any $\widetilde \gamma>\max\{\gamma+\frac12,\frac32\}$,
\begin{align}
\big\|[J^\gamma, f]\partial_x g\big\|_{L^2}\lesssim \|f\|_{H^{\widetilde{\gamma}}}\|g\|_{H^\gamma},\label{20.45}
\end{align}
and for any $\gamma>\frac32$,
\begin{align}
\big\|[J^\gamma, f]\partial_x g\big\|_{L^2}\lesssim \|f\|_{H^\gamma}\|g\|_{H^\gamma}.\label{20.46}
\end{align}
\end{remark}

Based on the above inequalities,
we can deduce some estimates as follows, which will be used to obtain the \emph{a prior} estimate of the numerical solution.
\begin{lemma}\label{lm3.5}
The following inequalities hold:
\begin{itemize}
  \item[(i)]
  For any $\gamma\ge 0, \gamma_1>\frac12$, $f\in H^{\gamma},g\in H^{\gamma+\gamma_1+1}$, then
\begin{align*}
\big\langle J^\gamma\partial_x(fg),J^\gamma f\big\rangle\lesssim \|f\|_{H^\gamma}^2\|g\|_{H^{\gamma+\gamma_1+1}}.
\end{align*}
  \item[(ii)]
  For any $\gamma>\frac32$, $f\in H^{\gamma}$, then
\begin{align*}
\big\langle J^\gamma\partial_x\big(f^2\big),J^\gamma f\big\rangle\lesssim \|f\|_{H^\gamma}^3.
\end{align*}
\end{itemize}
\end{lemma}

\begin{proof}
(i) Directly, we have
\begin{align*}
\big\langle J^\gamma\partial_x(fg),J^\gamma f\big\rangle=\big\langle J^\gamma \partial_x f\cdot g,J^\gamma f\big\rangle
+\big\langle J^\gamma\big(f\cdot \partial_xg\big),J^\gamma f\big\rangle
+\big\langle\big[J^\gamma,g\big]\>\partial_x f,J^\gamma f\big\rangle.
\end{align*}
For the first term on the right-hand side, by using integration-by-parts, it is equal to
\begin{align*}
-\frac12\int \partial_xg\big|J^\gamma f\big|^2\,dx.
\end{align*}
Therefore, we have the estimate
\begin{align*}
\big|\big\langle J^\gamma \partial_x f\cdot g,J^\gamma f\big\rangle \big|
\lesssim  \big\|g\big\|_{H^{\gamma_1+1}}\big\|f\big\|_{H^\gamma}^2,
\end{align*}
for any $\gamma_1>\frac12$. For the second term, by Lemma \ref{lem:kato-Ponce} (ii), we have
\begin{align*}
\big|\big\langle J^\gamma\big(f\cdot \partial_xg\big),J^\gamma f\big\rangle\big|
\lesssim &  \big\|f\cdot \partial_xg\big\|_{H^\gamma} \big\| f\big\|_{H^\gamma}\\
\lesssim & \big\|g\big\|_{H^{\gamma+\gamma_1+1}}\|f\|_{H^\gamma}^2 .
\end{align*}
For the third term, by \eqref{20.45} we have
\begin{align*}
\big|\big\langle\big[J^\gamma,g\big]\>\partial_x f,J^\gamma f\big\rangle\big|
\lesssim &  \big\|\big[J^\gamma,g\big]\>\partial_x  f\big\|_{L^2} \big\| f\big\|_{H^\gamma}
\lesssim   \big\|f\big\|_{H^\gamma}^2\big\|g\big\|_{H^{\gamma+\gamma_1+1}}.
\end{align*}
Combining the three estimates above, we get the estimate in (i).

(ii) We use the similar argument to write
$$
\big\langle J^\gamma\partial_x\big(f^2\big),J^\gamma f\big\rangle
=2\big\langle J^\gamma \partial_x f\cdot f,J^\gamma f\big\rangle
+2\big\langle\big[J^\gamma,f\big]\>\partial_x f,J^\gamma f\big\rangle,
$$
and then for the first term on the right-hand side, we get for any $\gamma_1>\frac{1}{2}$,
\begin{align*}
\big|\big\langle J^\gamma \partial_x f\cdot f,J^\gamma f\big\rangle \big|
\lesssim  \|f\|_{H^{\gamma_1+1}}\|f\|_{H^\gamma}^2.
\end{align*}
By choosing $\gamma_1$ properly, we have $\|f\|_{H^{\gamma_1+1}}\lesssim\|f\|_{H^\gamma}$.
For the second term, by applying \eqref{20.46} instead, we get
\begin{align*}
\big|\big\langle\big[J^\gamma,f\big]\>\partial_x f,J^\gamma f\big\rangle \big|
\lesssim \big\|f\big\|_{H^\gamma}^3,
\end{align*}
and hence, we obtain the estimate in (ii).
\end{proof}

\subsection{Problem reduction}\label{sub1p5}
Now, we start to illustrate the proof of the convergence theorem.
For the simplicity of notations, we shall omit the space variable $x$ in the functions, and we define
\begin{align}
v^n:=\fe^{t_n\partial_x^3}u^n\quad \mbox{ and } \quad {v}_t^{n}:=\frac12\fe^{t_n\partial_x^3}
\partial_x\left(\fe^{-t_n\partial_x^3}v^n\right)^2,\label{vn-t}
\end{align}
where $u^n$ is the numerical solution from the LRI scheme (\ref{LRI}).
Noting that the scheme \eqref{LRI} is obtained by exactly integrating (\ref{v app}), so we have
\begin{align}
v^{n+1}
 =&v^n+
 \frac12\int_0^\tau\fe^{(t_n+s)\partial_x^3}
 \partial_x\left(\fe^{-(t_n+s)\partial_x^3}v^n\right)^2ds\notag\\
 &+\int_0^\tau s\fe^{(t_n+s)\partial_x^3}
 \partial_x\left[
 \left(\fe^{-(t_n+s)\partial_x^3}v^n\right)
 \left(\fe^{-(t_n+s)\partial_x^3}{v}_t^{n}\right)\right]ds.\label{LRI-discrete}
\end{align}
Since the operator $\fe^{t\partial_x^3}$ in the change of variable (\ref{twist}) is unitary, so to prove (\ref{main result}), it is sufficient to prove
\begin{equation*}
  \|v(t_n)-v^n\|_{H^\gamma}\leq C\tau^2,\quad n=0,1,\ldots,\frac{T}{\tau}.
\end{equation*}
To do this, we subtract (\ref{LRI-discrete}) from the exact Duhamel's formula (\ref{v eq}) to get:
\begin{align}\label{11.22}
v(t_{n+1})-v^{n+1}
=&\mathcal{L}^n+\Phi^n\left(v(t_n)\right)-\Phi^n\left(v^n\right),
\end{align}
where we define the local error term as
\begin{align}
\mathcal{L}^n:=&\frac{1}{2}\int_0^\tau \fe^{(t_n+s)\partial_x^3}\partial_x\left(\fe^{-(t_n+s)\partial_x^3}v(t_n+s)\right)^2\,ds
 -\frac12\int_0^\tau\fe^{(t_n+s)\partial_x^3}
 \partial_x\left(\fe^{-(t_n+s)\partial_x^3}v(t_n)\right)^2ds\notag\\
 & -\int_0^\tau s\fe^{(t_n+s)\partial_x^3}
 \partial_x\left[
 \left(\fe^{-(t_n+s)\partial_x^3}v(t_n)\right)
 \left(\fe^{-(t_n+s)\partial_x^3}\partial_tv(t_n)\right)\right]ds,
\label{11.26}
\end{align}
and the numerical propagator as
\begin{align}
\Phi^n(v):=&v+\frac12\int_0^\tau\fe^{(t_n+s)\partial_x^3}
 \partial_x\left(\fe^{-(t_n+s)\partial_x^3}v\right)^2ds\notag\\
 & +\frac12\int_0^\tau s\fe^{(t_n+s)\partial_x^3}
 \partial_x\left[
 \left(\fe^{-(t_n+s)\partial_x^3}v\right)
\left(\fe^{-s\partial_x^3}
\partial_x\left(\fe^{-t_n\partial_x^3}v\right)^2\right)\right]ds
\label{phin def}.
\end{align}
Hence, to obtain a Gronwall type inequality, it reduces to control $\mathcal{L}^n$ and $\Phi\big(v(t_n)\big)-\Phi\big(v^n\big)$, which are regarded as the local error estimate and the stability in the following.


By directly  calculations, we have the following key facts.
\begin{lemma}\label{lem:two-form}
The following equalities hold:

(i) Let $v$ be the solution of \eqref{vt eq}, then
  $$
\partial_t^2v(t)=\frac32\fe^{t\partial_x^3}\partial_{x}^2
\left(\fe^{-t\partial_x^3}\partial_xv(t)\right)^2
+\frac13\fe^{t\partial_x^3}\partial_{x}^2
\left(\fe^{-t\partial_x^3}v(t)\right)^3,\quad t\geq0.
$$

 (ii) Let $f,g\in L^2$ with $\widehat{f}(0)=\widehat{g}(0)=0$, then for any $t_n\geq0$,
  \begin{align*}
   &\int_0^\tau \fe^{(t_n+s)\partial_x^3}\partial_x\left(\fe^{-(t_n+s)\partial_x^3}f\cdot \fe^{-(t_n+s)\partial_x^3}g\right)\,ds\\
  =&
  \frac13 \fe^{t_{n+1}\partial_x^3}\left(\fe^{-t_{n+1}\partial_x^3}\partial_x^{-1}f\cdot \fe^{-t_{n+1}\partial_x^3}\partial_x^{-1}g\right)
  -\frac13\fe^{t_{n}\partial_x^3}\left(\fe^{-t_{n}\partial_x^3}\partial_x^{-1}f\cdot \fe^{-t_{n}\partial_x^3}\partial_x^{-1}g\right);
  \end{align*}
 Moreover for $k=1,2,\ldots,$
  \begin{align*}
  &\int_0^\tau s^k\fe^{(t_n+s)\partial_x^3}\partial_x\left(\fe^{-(t_n+s)\partial_x^3}f\cdot \fe^{-(t_n+s)\partial_x^3}g\right)\,ds\\
  =&\frac{\tau^k}3 \fe^{t_{n+1}\partial_x^3}\left(\fe^{-t_{n+1}\partial_x^3}\partial_x^{-1}f\cdot \fe^{-t_{n+1}\partial_x^3}\partial_x^{-1}g\right)\\
  &-\frac{k}3\int_0^\tau\! s^{k-1}\fe^{(t_n+s)\partial_x^3}\left(\fe^{-(t_n+s)\partial_x^3}\partial_x^{-1}f\cdot \fe^{-(t_n+s)\partial_x^3}\partial_x^{-1}g\right)\,ds.
  \end{align*}
\end{lemma}

\begin{proof}
(i) Noting that
\begin{align}
\partial_t\widehat{v}(t,\xi)=\frac12i\xi\int_{\xi=\xi_1+\xi_2}
\fe^{-it\left(\xi^3-\xi_1^3-\xi_2^3\right)}\widehat v(\xi_1) \widehat v(\xi_2)\,(d\xi_1),\quad t\geq0,\label{pt-v-F}
\end{align}
and
\begin{align*}
\xi^3-\xi_1^3-\xi_2^3=3\xi\xi_1\xi_2, 
\end{align*}
we have that for any $t\ge 0$,
\begin{align*}
\partial_t^2\widehat{v}(t,\xi)=&\frac32\xi^2\int_{\xi=\xi_1+\xi_2}
\fe^{-it\left(\xi^3-\xi_1^3-\xi_2^3\right)}\xi_1\xi_2\widehat v(\xi_1) \widehat v(\xi_2)\,(d\xi_1)\\
&+\frac12i\xi\int_{\xi=\xi_1+\xi_2}
\fe^{-it\left(\xi^3-\xi_1^3-\xi_2^3\right)}\partial_t\big(\widehat v(\xi_1) \widehat v(\xi_2)\big)\,(d\xi_1).
\end{align*}
From \eqref{pt-v-F} and symmetry, we get
\begin{align*}
\partial_t^2\widehat{v}(t,\xi)=&\frac32\xi^2\int_{\xi=\xi_1+\xi_2}
\fe^{-it\left(\xi^3-\xi_1^3-\xi_2^3\right)}\xi_1\xi_2\widehat v(\xi_1) \widehat v(\xi_2)\,(d\xi_1)\\
&+i\xi\int_{\xi=\xi_1+\xi_2+\xi_3}
\frac12i(\xi_1+\xi_2)\>\fe^{-it\left(\xi^3-\xi_1^3-\xi_2^3-\xi_3^3\right)}\widehat v(\xi_1) \widehat v(\xi_2)\widehat v(\xi_3)\,(d\xi_1)(d\xi_2).
\end{align*}
By symmetry again, the second term is equal to
\begin{align*}
\quad&-\frac13\xi\int_{\xi=\xi_1+\xi_2+\xi_3}
(\xi_1+\xi_2+\xi_3)\>\fe^{-it\left(\xi^3-\xi_1^3-\xi_2^3-\xi_3^3\right)}\widehat v(\xi_1) \widehat v(\xi_2)\widehat v(\xi_3)\,(d\xi_1)(d\xi_2)\\
=&-\frac13\xi^2\int_{\xi=\xi_1+\xi_2+\xi_3}
\fe^{-it\left(\xi^3-\xi_1^3-\xi_2^3-\xi_3^3\right)}\widehat v(\xi_1) \widehat v(\xi_2)\widehat v(\xi_3)\,(d\xi_1)(d\xi_2).
\end{align*}
Hence, we obtain that
\begin{align*}
\partial_t^2\widehat{v}(t,\xi)=&\frac32\xi^2\int_{\xi=\xi_1+\xi_2}
\fe^{-it\left(\xi^3-\xi_1^3-\xi_2^3\right)}\xi_1\xi_2\widehat v(\xi_1) \widehat v(\xi_2)\,(d\xi_1)\\
&-\frac13\xi^2\int_{\xi=\xi_1+\xi_2+\xi_3}
\fe^{-it\left(\xi^3-\xi_1^3-\xi_2^3-\xi_3^3\right)}\widehat v(\xi_1) \widehat v(\xi_2)\widehat v(\xi_3)\,(d\xi_1)(d\xi_2).
\end{align*}
This proves the equality in (i) by the inverse Fourier transform.

(ii) For $k\geq0$, by taking the Fourier transform we get for any $t_n\geq0$,
  \begin{align*}
  \quad&\mathcal F\left(\int_0^\tau s^k\fe^{(t_n+s)\partial_x^3}\partial_x\big(\fe^{-(t_n+s)\partial_x^3}f\cdot \fe^{-(t_n+s)\partial_x^3}g\big)\,ds\right)(\xi)\\
  =&i\xi\int_0^\tau\!\int_{\xi=\xi_1+\xi_2}\! s^k\fe^{-i(t_n+s)\left(\xi^3-\xi_1^3-\xi_2^3\right)}\widehat{f}(\xi_1)\widehat{g}(\xi_2)\,(d\xi_1)ds.
  \end{align*}
  Note that for $k=0$,
   $$
  \int_0^\tau\!\fe^{-i(t_n+s)\left(\xi^3-\xi_1^3-\xi_2^3\right)}ds
  =-\frac{1}{3i\xi\xi_1\xi_2}\fe^{-it_{n+1}\left(\xi^3-\xi_1^3-\xi_2^3\right)}
  +\frac{1}{3i\xi\xi_1\xi_2}\fe^{-it_{n}\left(\xi^3-\xi_1^3-\xi_2^3\right)},
  $$
  and for $k\geq1$,
 \begin{align*}
  &\int_0^\tau\!s^k\fe^{-i(t_n+s)\left(\xi^3-\xi_1^3-\xi_2^3\right)}ds\\
  =&
  -\frac{\tau^k}{3i\xi\xi_1\xi_2}\fe^{-it_{n+1}\left(\xi^3-\xi_1^3-\xi_2^3\right)}
  +\frac{k}{3i\xi\xi_1\xi_2}\int_0^\tau\! s^{k-1}\fe^{-i(t_n+s)
  \left(\xi^3-\xi_1^3-\xi_2^3\right)}ds.
  \end{align*}
Then by the above formulas, we find
  \begin{align*}
  \quad&\mathcal F\left(\int_0^\tau \fe^{(t_n+s)\partial_x^3}\partial_x\left[\fe^{-(t_n+s)\partial_x^3}f\cdot \fe^{-(t_n+s)\partial_x^3}g\right]\,ds\right)(\xi)\\
  =&-\int_{\xi=\xi_1+\xi_2}\! \frac{1}{3\xi_1\xi_2}\fe^{-it_{n+1}\left(\xi^3-\xi_1^3-\xi_2^3\right)}
  \widehat{f}(\xi_1)\widehat{g}(\xi_2)\,(d\xi_1)\\
 & +\int_{\xi=\xi_1+\xi_2}\! \frac{1}{3\xi_1\xi_2}\fe^{-it_{n}\left(\xi^3-\xi_1^3-\xi_2^3\right)}
 \widehat{f}(\xi_1)\widehat{g}(\xi_2)\,(d\xi_1),
  \end{align*}
  and for $k\geq1$,
  \begin{align*}
  \quad&\mathcal F\left(\int_0^\tau s^k\fe^{(t_n+s)\partial_x^3}\partial_x\left[\fe^{-(t_n+s)\partial_x^3}f\cdot \fe^{-(t_n+s)\partial_x^3}g\right]\,ds\right)(\xi)\\
  =&-\tau^k\int_{\xi=\xi_1+\xi_2}\! \frac{1}{3\xi_1\xi_2}\fe^{-it_{n+1}\left(\xi^3-\xi_1^3-\xi_2^3\right)}
  \widehat{f}(\xi_1)\widehat{g}(\xi_2)\,(d\xi_1)\\
 & +k\int_0^\tau\!\int_{\xi=\xi_1+\xi_2}\! \frac{s^{k-1}}{3\xi_1\xi_2}\fe^{-i(t_n+s)\left(\xi^3-\xi_1^3-\xi_2^3\right)}
 \widehat{f}(\xi_1)\widehat{g}(\xi_2)\,(d\xi_1)ds,
  \end{align*}
 which give the two equalities in (ii) by the inverse Fourier transform.
\end{proof}

Some consequences of the above formulas together with the Kato-Ponce inequality are the following two lemmas, which will be used for the proof of the boundedness of the  numerical solution.
\begin{lemma}\label{lem:bi-est}  Let $f\in H^\gamma,g\in H^{\gamma+\gamma_1}$ with $\widehat{f}(0)=\widehat{g}(0)=0$ for $\gamma\ge 0, \gamma_1>\frac12$, then the following inequality holds for any $t_n\geq0$:
\begin{align*}
\Big\|\int_0^\tau \fe^{(t_n+s)\partial_x^3}&\partial_x\big(\fe^{-(t_n+s)\partial_x^3}f\cdot \fe^{-(t_n+s)\partial_x^3}g\big)\,ds\Big\|_{H^\gamma}
\lesssim \sqrt\tau\|f\|_{H^\gamma}\|g\|_{H^{\gamma+\gamma_1}}.
\end{align*}
Moreover, if $\gamma>\frac12$, then
\begin{align}
\Big\|\int_0^\tau \fe^{(t_n+s)\partial_x^3}&\partial_x\big(\fe^{-(t_n+s)\partial_x^3}f\cdot \fe^{-(t_n+s)\partial_x^3}g\big)\,ds\Big\|_{H^\gamma}
\lesssim \sqrt\tau\|f\|_{H^\gamma}\|g\|_{H^{\gamma}}.\label{bi-est-2}
\end{align}
\end{lemma}

\begin{proof}
From Lemma \ref{lem:two-form}-(ii) and integration-by-parts, we get for any $t_n\geq0$ and $\gamma\geq0$,
\begin{align}
\quad&\Big\|\int_0^\tau  \fe^{(t_n+s)\partial_x^3}\partial_x\big(\fe^{-(t_n+s)\partial_x^3}f\cdot \fe^{-(t_n+s)\partial_x^3}g\big)\,ds\Big\|_{H^\gamma}^2\label{15.48}\\
=&\frac13\int_0^\tau\!\! \Big\langle J^\gamma\fe^{(t_n+s)\partial_x^3}\big(\fe^{-(t_n+s)\partial_x^3}f\cdot \fe^{-(t_n+s)\partial_x^3}g\big), \fe^{t_n\partial_x^3}\partial_xJ^\gamma\big(\fe^{-t_n\partial_x^3}\partial_x^{-1}f\cdot \fe^{-t_n\partial_x^3}\partial_x^{-1}g\big)\Big\rangle\,ds\nonumber\\
&-\frac13\int_0^\tau\!\! \Big\langle J^\gamma\fe^{(t_n+s)\partial_x^3}\big(\fe^{-(t_n+s)\partial_x^3}f\cdot \fe^{-(t_n+s)\partial_x^3}g\big), \fe^{t_{n+1}\partial_x^3}\partial_xJ^\gamma
\big(\fe^{-t_{n+1}\partial_x^3}\partial_x^{-1}f\cdot \fe^{-t_{n+1}\partial_x^3}\partial_x^{-1}g\big)\Big\rangle\,ds.\notag
\end{align}
 By  Cauchy-Schwarz's inequality,
\begin{align*}
\big|\eqref{15.48}\big|
\lesssim &
\int_0^\tau \Big\| J^\gamma\big(\fe^{-(t_n+s)\partial_x^3}f\cdot \fe^{-(t_n+s)\partial_x^3}g\big)\Big\|_{L^2} \Big\|\partial_xJ^\gamma\big(\fe^{-t_{n}\partial_x^3}\partial_x^{-1}f\cdot \fe^{-t_{n}\partial_x^3}\partial_x^{-1}g\big)\Big\|_{L^2}\,ds\\
&+
\int_0^\tau \Big\| J^\gamma\big(\fe^{-(t_n+s)\partial_x^3}f\cdot \fe^{-(t_n+s)\partial_x^3}g\big)\Big\|_{L^2} \Big\|\partial_xJ^\gamma\big(\fe^{-t_{n+1}\partial_x^3}\partial_x^{-1}f\cdot \fe^{-t_{n+1}\partial_x^3}\partial_x^{-1}g\big)\Big\|_{L^2}\,ds.
\end{align*}
For simplicity, we shall only present the estimate of the second term on the right-hand side of the above inequality, and the first term can be treated in the same way.
By  Lemma \ref{lem:kato-Ponce} (ii), we have for any $\gamma\geq0$,
\begin{align*}
\Big\| J^\gamma\big(\fe^{-(t_n+s)\partial_x^3}f\cdot \fe^{-(t_n+s)\partial_x^3}g\big)\Big\|_{L^2}
\lesssim &\| f\|_{H^{\gamma}}\| g\|_{H^{\gamma+\gamma_1}};
\end{align*}
Or by Lemma \ref{lem:kato-Ponce} (i), when $\gamma>\frac12$,
\begin{align*}
\Big\| J^\gamma\big(\fe^{-(t_n+s)\partial_x^3}f\cdot \fe^{-(t_n+s)\partial_x^3}g\big)\Big\|_{L^2}
\lesssim & \|f\|_{H^{\gamma}}\| g\|_{H^{\gamma}}.
\end{align*}
Similarly, we have for any  $\gamma\geq0,\,\gamma_1>\frac12$,
\begin{align*}
\Big\|\partial_xJ^\gamma\big(\fe^{-t_{n+1}\partial_x^3}\partial_x^{-1}f\cdot \fe^{-t_{n+1}\partial_x^3}\partial_x^{-1}g\big)\Big\|_{L^2}
\lesssim \|f\|_{H^{\gamma}}\| g\|_{H^{\gamma+\gamma_1}},
\end{align*}
and for $\gamma>\frac12$,
\begin{align*}
\Big\|\partial_xJ^\gamma\big(\fe^{-t_{n+1}\partial_x^3}\partial_x^{-1}f\cdot \fe^{-t_{n+1}\partial_x^3}\partial_x^{-1}g\big)\Big\|_{L^2}
\lesssim \|f\|_{H^{\gamma}}\| g\|_{H^{\gamma}}.
\end{align*}
Therefore, in total we find that for $\gamma\geq0,\,\gamma_1>\frac12$,
\begin{align*}
\big|\eqref{15.48}\big|
\lesssim \tau\| f\|_{H^{\gamma}}^2\| g\|_{H^{\gamma+\gamma_1}}^2;
\end{align*}
Or when $\gamma>\frac12$,
\begin{align*}
\big|\eqref{15.48}\big|
\lesssim \tau\| f\|_{H^{\gamma}}^2\| g\|_{H^{\gamma}}^2.
\end{align*}
This finishes the proof of the lemma.
\end{proof}
Moreover, we have
\begin{lemma}\label{lem:tri-est} The following estimates hold:
\begin{itemize}
  \item[(i)]
  Let $f_1\in H^{\gamma}, f_2\in H^{\gamma+\gamma_1}$ and $f_3\in H^{\gamma+\gamma_1}$ for $\gamma\ge0, \gamma_1>\frac12$ with $\widehat{f_j}(0)=0$ for $j=1,2,3$, then for any $t_n\geq0$, $t\in \R$ and $k\geq1$,
$$
\Big\|\int_0^\tau s^k\fe^{(t_n+s)\partial_x^3}\partial_x\Big(\fe^{-(t+s)\partial_x^3}\partial_x\big(f_1f_2\big)\cdot \fe^{-(t_n+s)\partial_x^3}f_3\Big)\,ds\Big\|_{H^{\gamma}}
\lesssim \tau^k\|f_1\|_{H^\gamma}\|f_2\|_{H^{\gamma+\gamma_1}}\|f_3\|_{H^{\gamma+\gamma_1-1}}.
$$
  \item[(ii)]
  Let $f_j\in H^{\gamma_0}$ for $\gamma_0>\frac12$ with $\widehat{f_j}(0)=0$ for $j=1,2,3$, then for any $t_n\geq0$, $t\in \R$ and $k\geq1$,
$$
\Big\|\int_0^\tau s^k\fe^{(t_n+s)\partial_x^3}\partial_x\Big(\fe^{-(t+s)\partial_x^3}
\partial_x\big(f_1f_2\big)\cdot \fe^{-(t_n+s)\partial_x^3}f_3\Big)\,ds\Big\|_{H^{\gamma_0}}
\lesssim \tau^k \|f_1\|_{H^{\gamma_0}}\|f_2\|_{H^{\gamma_0}}\|f_3\|_{H^{\gamma_0-1}}.
$$
\end{itemize}
\end{lemma}
\begin{proof}
(i) From Lemma \ref{lem:two-form}-(ii), we find for $k\geq1$,
\begin{align*}
\quad &\int_0^\tau s^k\fe^{(t_n+s)\partial_x^3}\partial_x\Big(\fe^{-(t+s)\partial_x^3}
\partial_x\big(f_1f_2\big)\cdot \fe^{-(t_n+s)\partial_x^3}f_3\Big)\,ds\\
=&\frac{\tau^k}{3}\fe^{t_{n+1}\partial_x^3}
\Big(\fe^{-(t+\tau)\partial_x^3}\big(f_1f_2\big)\cdot \fe^{-(t_n+\tau)\partial_x^3}\partial_x^{-1}f_3\Big)\\
& -\frac k3\int_0^\tau s^{k-1}\fe^{(t_{n}+s)\partial_x^3}\Big(\fe^{-(t+s)\partial_x^3}\big(f_1f_2\big)\cdot \fe^{-(t_{n}+s)\partial_x^3}\partial_x^{-1}f_3\Big)\,ds.
\end{align*}
By H\"older's inequality and Sobolev's inequality, we have for any $\gamma\geq0,\gamma_1>\frac12$,
\begin{align*}
\quad &\Big\|\int_0^\tau s^k\fe^{(t_n+s)\partial_x^3}\partial_x
\Big(\fe^{-(t+s)\partial_x^3}\partial_x\big(f_1f_2\big)\cdot \fe^{-(t_n+s)\partial_x^3}f_3\Big)\,ds\Big\|_{H^\gamma}\\
\lesssim & \tau^k\Big(\|f_1\|_{H^\gamma}\|f_2\|_{H^{\gamma_1}}
\|\partial_x^{-1}f_3\|_{H^{\gamma_1}}
+\|f_1\|_{L^2}\|f_2\|_{H^{\gamma+\gamma_1}}\|\partial_x^{-1}f_3\|_{H^{\gamma_1}}\\
&+\|f_1\|_{L^2}\|f_2\|_{H^{\gamma_1}}\|\partial_x^{-1}f_3\|_{H^{\gamma+\gamma_1}}\Big)\\
\lesssim&
\tau^k\|f_1\|_{H^\gamma}\|f_2\|_{H^{\gamma+\gamma_1}}\|f_3\|_{H^{\gamma+\gamma_1-1}}.
\end{align*}

(ii) By similar arguments as above but with the different H\"older's inequality, we have that for any $\gamma_1>\frac12$,
\begin{align*}
\quad &\Big\|\int_0^\tau s^k\fe^{(t_n+s)\partial_x^3}\partial_x\Big(\fe^{-(t+s)\partial_x^3}
\partial_x\big(f_1f_2\big)\cdot e^{-(t_n+s)\partial_x^3}f_3\Big)\,ds\Big\|_{H^{\gamma_0}}\\
\lesssim& \tau^k\Big(\|f_1\|_{H^{\gamma_0}}\|f_2\|_{H^{\gamma_1}}\|\partial_x^{-1}f_3\|_{H^{\gamma_1}}
+\|f_1\|_{H^{\gamma_1}}\|f_2\|_{H^{\gamma_0}}\|\partial_x^{-1}f_3\|_{H^{\gamma_1}}\\
&+\|f_1\|_{H^{\gamma_1}}\|f_2\|_{H^{\gamma_1}}\|\partial_x^{-1}f_3\|_{H^{\gamma_0}}\Big)\\
\lesssim&
\tau^k\|f_1\|_{H^{\gamma_0}}\|f_2\|_{H^{\gamma_0}}\|f_3\|_{H^{\gamma_0-1}},
\end{align*}
where we used the fact $\gamma_0>\frac12$ in the last step.
\end{proof}

\subsection{\emph{A priori} estimate}
With the prepared lemmas before, we can obtain the \emph{a priori} estimate of the numerical solution $v^n$ which will be a key for the stability proof later. It is done here by establishing a weaker convergence rate of the scheme as in \cite{Lubich} together with estimates from the Kato-Ponce inequaltiy.

\begin{lemma}\label{lem:est-vn} (A priori estimate of $v^n$)
  For any $\gamma_0>\frac32$, if $v_0\in H^{\gamma_0+2}$, then there exist constants $\tau_0>0$  and $C>0$, such that for any $0<\tau\leq\tau_0$ we have
$$
\|v^{n}\|_{H^{\gamma_0}}\le C,\quad n=0,1,\ldots,\frac{T}{\tau},
$$
 where $\tau_0$ and $C$ depend only on $T$ and $\|v\|_{L^\infty((0,T);H^{\gamma_0+2})}$.
\end{lemma}
\begin{proof}
The proof goes in the manner of  bootstrap argument by assuming that $v^n\in H^{\gamma_0}$ for some $0\leq n\leq \frac{T}{\tau}$.
Taking the difference between \eqref{LRI-discrete} and the exact Duhamel's formula (\ref{v eq}), we have
\begin{align*}
v^{n+1}-v(t_{n+1})
 =v^n-v(t_n)+L_1+L_2,\quad n=0,\ldots,\frac{T}{\tau}-1,
\end{align*}
where  we denote
\begin{align*}
 L_1=&\frac{1}{2}\int_0^\tau
 \fe^{(t_n+s)\partial_x^3}\partial_x\left[\Big(\fe^{-(t_n+s)\partial_x^3}
 v^n\Big)^2-\Big(\fe^{-(t_n+s)\partial_x^3}
 v(t_n+s)\Big)^2\right]ds,\\
 L_2=&\int_0^\tau s
 \fe^{(t_n+s)\partial_x^3}\partial_x\left[\fe^{-(t_n+s)\partial_x^3}
 v^n\cdot \fe^{-(t_n+s)\partial_x^3}{v}_t^{n}\right]ds.
\end{align*}
Thus, we get that
\begin{align}
\big\|v^{n+1}-v(t_{n+1})\big\|_{H^{\gamma_0}}^2\le& \big\|v^n-v(t_n)\big\|_{H^{\gamma_0}}^2+
2\big\langle J^{\gamma_0}(v^n-v(t_n)), J^{\gamma_0}L_1\big\rangle\notag\\
 &+
2\big\langle J^{\gamma_0}(v^n-v(t_n)), J^{\gamma_0}L_2\big\rangle
+2\big\|L_1\big\|_{H^{\gamma_0}}^2+2\big\|L_2\big\|_{H^{\gamma_0}}^2.\label{vn-vtn-iteration}
\end{align}
In the following, we shall give estimate of the right-hand side of (\ref{vn-vtn-iteration}) term by term.

Firstly, we decompose $L_1$ into two parts as
\begin{align}
L_1
 =&\frac{1}{2}\int_0^\tau
 \fe^{(t_n+s)\partial_x^3}\partial_x\left[\Big(\fe^{-(t_n+s)\partial_x^3}
 v^n\Big)^2-\Big(\fe^{-(t_n+s)\partial_x^3}
 v(t_n)\Big)^2\right]ds\notag\\
 &+\frac{1}{2}\int_0^\tau
 \fe^{(t_n+s)\partial_x^3}\partial_x\left[\Big(\fe^{-(t_n+s)\partial_x^3}
 v(t_n)\Big)^2-\Big(\fe^{-(t_n+s)\partial_x^3}
 v(t_n+s)\Big)^2\right]ds\notag\\
 =:&L_{11}+L_{12}.\label{13.00}
\end{align}
Then we write
$$
\big\langle J^{\gamma_0} (v^n-v(t_n)), J^{\gamma_0}L_1\big\rangle
=\big\langle J^{\gamma_0} (v^n-v(t_n)), J^{\gamma_0}L_{11}\big\rangle+
\big\langle J^{\gamma_0} (v^n-v(t_n)), J^{\gamma_0}L_{12}\big\rangle.
$$
For the first part, we have
\begin{subequations}
\begin{align}
&\quad2\big\langle J^{\gamma_0} (v^n-v(t_n)), J^{\gamma_0}L_{11}\big\rangle\notag\\
 &=\int_0^\tau\left\langle J^{\gamma_0}\fe^{-(t_n+s)\partial_x^3}(v^n-v(t_n)),
J^{\gamma_0}\partial_x\left[\left(\fe^{-(t_n+s)\partial_x^3}
 \left(v^n-v(t_n)\right)\right)^2\right]\right\rangle ds \label{9.04-1}\\
&\ \,+2\int_0^\tau\Big\langle J^{\gamma_0}\fe^{-(t_n+s)\partial_x^3}(v^n-v(t_n)),
J^{\gamma_0}\partial_x\left[\fe^{-(t_n+s)\partial_x^3}\big(v^n-v(t_n)\big)\cdot\fe^{-(t_n+s)\partial_x^3}
 v(t_n)\right]\Big\rangle ds. \label{9.04-2}
\end{align}
\end{subequations}
For \eqref{9.04-1}, using Lemma \ref{lm3.5}-(ii), we get that for any $\gamma_0>\frac32$,
\begin{align*}
\big|\eqref{9.04-1}\big|\lesssim \tau \big\|v^n-v(t_n)\big\|_{H^{\gamma_0}}^3.
\end{align*}
For \eqref{9.04-2}, using Lemma \ref{lm3.5}-(i) instead,  we get that for any $\gamma\geq0,\gamma_1>\frac12$,
\begin{align*}
\big|\eqref{9.04-2}\big|\lesssim \tau \left\|v(t_n)\right\|_{H^{\gamma+\gamma_1+1}}
\left\|v^n-v(t_n)\right\|_{H^{\gamma}}^2,
\end{align*}
and then for any $\gamma_0>\frac32$, by properly choosing the $\gamma,\gamma_1$ and the assumption of the lemma with $\gamma_0=\gamma+\gamma_1+1$, we have
$$\|v(t_n)\|_{H^{\gamma+\gamma_1+1}}\lesssim1.$$
Hence, in total we obtain that
\begin{align}
\big|\big\langle J^{\gamma_0} (v^n-v(t_n)), J^{\gamma_0}L_{11}\big\rangle\big|
\le C\tau \Big(\big\|v^n-v(t_n)\big\|_{H^{\gamma_0}}^2+ \big\|v^n-v(t_n)\big\|_{H^{\gamma_0}}^3\Big),\label{13.12}
\end{align}
where the constant $C>0$ depends only on $\|v\|_{L^\infty((t_n,t_{n+1});H^{\gamma_0+\gamma_1+1})}$.

For $\big\langle J^{\gamma_0} (v^n-v(t_n)), J^{\gamma_0}L_{12}\big\rangle$, we  claim that
\begin{align}\label{L12}
\big\|L_{12}\big\|_{H^{\gamma_0}}\lesssim \tau^2 \|v\|_{L^\infty((t_n,t_{n+1});H^{\gamma_0+2})}^3.
\end{align}
Indeed,
using Lemma \ref{lem:kato-Ponce} (i),  we get that
\begin{align*}
\big\|L_{12}\big\|_{H^{\gamma_0}}
\le &\frac12 \int_0^\tau \Big\|J^{\gamma_0}\partial_x\left[\fe^{-(t_n+s)\partial_x^3}\big(v(t_n)-v(t_n+s)\big)\cdot\fe^{-(t_n+s)\partial_x^3}
 \big(v(t_n)+v(t_n+s)\big)\right]\Big\|_{L^2}\,ds\\
 \lesssim &  \int_0^\tau \Big\|J^{\gamma_0+1}\left[\fe^{-(t_n+s)\partial_x^3}\big(v(t_n)-v(t_n+s)\big)\cdot\fe^{-(t_n+s)\partial_x^3}
 \big(v(t_n)+v(t_n+s)\big)\right]\Big\|_{L^2}\,ds\\
 \lesssim & \tau \big\|v(t_n+s)-v(t_n)\big\|_{L^\infty((0,\tau);H^{\gamma_0+1})}\|v\|_{L^\infty((t_n,t_{n+1});H^{\gamma_0+1})}.
\end{align*}
Note that
\begin{align*}
\big\|v(t_n+s)-v(t_n)\big\|_{L^\infty((0,\tau);H^{\gamma_0+1})}
=&\Big\|\int_0^s \partial_tv(t_n+t)\,dt\Big\|_{L^\infty((0,\tau);H^{\gamma_0+1})}\\
\le & \tau\big\|\partial_tv(t)\big\|_{L^\infty((t_n,t_{n+1});H^{\gamma_0+1})}.
\end{align*}
Now we need the estimate on $\partial_tv(t_n)$. From the definition \eqref{vt eq}, using  Lemma \ref{lem:kato-Ponce} (i), we have that for any $\gamma\geq0$,
\begin{align}
\|\partial_tv(t_n)\|_{H^{\gamma}}
&\lesssim\big\|\big(\fe^{-t_n\partial_x^3}v(t_n)\big)^2\big\|_{H^{\gamma+1}}
\lesssim\big\|v(t_n)\big\|_{H^{\gamma+1}}^2.\label{pt-v}
\end{align}
Using \eqref{pt-v},  we have
\begin{align*}
\big\|v(t_n+s)-v(t_n)\big\|_{L^\infty((0,\tau);H^{\gamma_0+1})}
\le & \tau\|v\|_{L^\infty((t_n,t_{n+1});H^{\gamma_0+2})}^2.
\end{align*}
Hence, we obtain \eqref{L12} and then we get
\begin{align*}
\big|\big\langle J^{\gamma_0} (v^n-v(t_n)), J^{\gamma_0}L_{12}\big\rangle\big|
 \lesssim & \tau^2 \big\|v^n-v(t_n)\big\|_{H^{\gamma_0}}\|v\|_{L^\infty((t_n,t_{n+1});H^{\gamma_0+2})}^3.
\end{align*}
The last estimate together with \eqref{13.12} and Cauchy-Schwartz's inequality,  we establish that
\begin{align}\label{L1-2}
\big|\big\langle J^{\gamma_0} (v^n-v(t_n)), J^{\gamma_0}L_1\big\rangle\big|
\le & C \tau \Big(\big\|v^n-v(t_n)\big\|_{H^{\gamma_0}}^2+ \big\|v^n-v(t_n)\big\|_{H^{\gamma_0}}^3\Big)+C\tau^2 \big\|v^n-v(t_n)\big\|_{H^{\gamma_0}}\notag\\
\le & C \tau \Big(\big\|v^n-v(t_n)\big\|_{H^{\gamma_0}}^2+ \big\|v^n-v(t_n)\big\|_{H^{\gamma_0}}^3\Big)+C\tau^3,
\end{align}
where the constant $C>0$ depends only on $\|v\|_{L^\infty((t_n,t_{n+1});H^{\gamma_0+2})}$.

Now we consider $\|L_j\|_{H^{\gamma_0}}$ for $j=1,2$ in (\ref{vn-vtn-iteration}). For $L_1$, from \eqref{13.00} and \eqref{L12}, we only need to consider $L_{11}$.
Indeed, from  \eqref{bi-est-2}, we have
\begin{align*}
\big\|L_{11}\big\|_{H^{\gamma_0}}
\lesssim & \sqrt\tau \big\|v^n-v(t_n)\big\|_{H^{\gamma_0}}\big\|v^n+v(t_n)\big\|_{H^{\gamma_0}}\\
\lesssim & \sqrt\tau \big\|v^n-v(t_n)\big\|_{H^{\gamma_0}}^2+\sqrt\tau \big\|v(t_n)\big\|_{H^{\gamma_0}}\big\|v^n-v(t_n)\big\|_{H^{\gamma_0}}.
\end{align*}
The above estimate together with \eqref{L12} give
\begin{align}\label{L1}
\big\|L_{1}\big\|_{H^{\gamma_0}}\le C\sqrt\tau \Big(\big\|v^n-v(t_n)\big\|_{H^{\gamma_0}}+\big\|v^n-v(t_n)\big\|_{H^{\gamma_0}}^2\Big)+C\tau^2,
\end{align}
where the constant $C>0$ depends only on $\|v\|_{L^\infty((t_n,t_{n+1});H^{\gamma_0+2})}$.

For $L_2$, we write
\begin{subequations}
\begin{align}
&L_2=\int_0^\tau s
 \fe^{(t_n+s)\partial_x^3}\partial_x\left[\fe^{-(t_n+s)\partial_x^3}
 v^n\cdot \fe^{-(t_n+s)\partial_x^3}{v}_t^{n}\right]ds\notag\\
 =&\int_0^\tau s
 \fe^{(t_n+s)\partial_x^3}\partial_x\left[\fe^{-(t_n+s)\partial_x^3}
 v^n\cdot \fe^{-(t_n+s)\partial_x^3}{v}_t^{n}-\fe^{-(t_n+s)\partial_x^3}
 v(t_n)\cdot \fe^{-(t_n+s)\partial_x^3}{\partial_tv}(t_n)\right]ds\label{16.28-1}\\
 &+\int_0^\tau s
 \fe^{(t_n+s)\partial_x^3}\partial_x\left[\fe^{-(t_n+s)\partial_x^3}
 v(t_n)\cdot \fe^{-(t_n+s)\partial_x^3}{\partial_tv}(t_n)\right]ds.\label{16.28-2}
\end{align}
\end{subequations}
For \eqref{16.28-1}, from \eqref{vt eq} and \eqref{vn-t}, and Lemma \ref{lem:tri-est}-(ii), we get
\begin{align*}
\big\|\eqref{16.28-1}\big\|_{H^{\gamma_0}}
\lesssim & \tau\big\|v^n-v(t_n)\big\|_{H^{\gamma_0}}\Big(\big\|v^n\big\|_{H^{\gamma_0}}^2+\big\|v(t_n)\big\|_{H^{\gamma_0}}^2\Big)\\
\lesssim &\tau\big\|v^n-v(t_n)\big\|_{H^{\gamma_0}}^3+\tau\big\|v^n-v(t_n)\big\|_{H^{\gamma_0}}\big\|v(t_n)\big\|_{H^{\gamma_0}}^2.
\end{align*}
For \eqref{16.28-2},  by  Lemma \ref{lem:kato-Ponce} (i), we get
\begin{align*}
\big\|\eqref{16.28-2}\big\|_{H^{\gamma_0}}
\lesssim & \tau^2\Big\|\partial_x\Big[\fe^{-(t_n+s)\partial_x^3}
 v(t_n)\cdot \fe^{-(t_n+s)\partial_x^3}{\partial_tv}(t_n)\Big]\Big\|_{H^{\gamma_0}}\\
\lesssim &\tau^2 \|v(t_n)\|_{H^{\gamma_0+1}}\|\partial_tv(t_n)\|_{H^{\gamma_0+1}}.
\end{align*}
Then using \eqref{pt-v}, we get
\begin{align*}
\big\|\eqref{16.28-2}\big\|_{H^{\gamma_0}}
\lesssim \tau^2\big\|v(t_n)\big\|_{H^{\gamma_0+2}}^3.
\end{align*}
Combining with these two estimates yields
\begin{align}\label{L2}
\big\|L_{2}\big\|_{H^{\gamma_0}}\le C\tau \Big(\big\|v^n-v(t_n)\big\|_{H^{\gamma_0}}+\big\|v^n-v(t_n)\big\|_{H^{\gamma_0}}^3\Big)+C\tau^2,
\end{align}
and thus by H\"older's and Cauchy-Schwartz's inequalities,
\begin{align}\label{L2-2}
\big\langle J^{\gamma_0}(v^n-v(t_n)), J^{\gamma_0}L_2\big\rangle
\le C\tau \Big(\big\|v^n-v(t_n)\big\|_{H^{\gamma_0}}^2+\big\|v^n-v(t_n)\big\|_{H^{\gamma_0}}^4\Big)+C\tau^3,
\end{align}
where the constant $C>0$ depends only on $\|v\|_{L^\infty((t_n,t_{n+1});H^{\gamma_0+2})}$.

Now inserting the estimates \eqref{L1-2}, \eqref{L1}, \eqref{L2} and \eqref{L2-2} into \eqref{vn-vtn-iteration}, we obtain that
\begin{align*}
\big\|v^{n+1}-v(t_{n+1})\big\|_{H^{\gamma_0}}^2\le \big\|v^n-v(t_n)\big\|_{H^{\gamma_0}}^2+C\tau \Big(\big\|v^n-v(t_n)\big\|_{H^{\gamma_0}}^2+\big\|v^n-v(t_n)\big\|_{H^{\gamma_0}}^4\Big)+C\tau^3.
\end{align*}
This implies that
\begin{align*}
\big\|v^{n+1}-v(t_{n+1})\big\|_{H^{\gamma_0}}\le (1+C\tau)\big\|v^n-v(t_n)\big\|_{H^{\gamma_0}}+C\tau \big\|v^n-v(t_n)\big\|_{H^{\gamma_0}}^2+C\tau^\frac32.
\end{align*}
Noting $v^0=v(0)$ and by Gronwall's inequality, we obtain that
\begin{align*}
\big\|v^{n+1}-v(t_{n+1})\big\|_{H^{\gamma_0}}\le C\tau^\frac32\sum\limits_{j=0}^{n+1}\big(1+C\tau\big)^j
\le C\tau^\frac12\big(1+C\tau\big)^{\frac T\tau}\le C\tau^\frac12 \fe^{CT},
\end{align*}
for $0\leq n\leq T/\tau-1$.
This proves the claimed result  of the lemma when $0<\tau\leq\tau_0$ for some $\tau_0$ depending on $T$ and $\|v\|_{L^\infty((0,T);H^{\gamma_0+2})}$.

\end{proof}

\subsection{Stability}\label{subsec3}
 Now we give stability result of the numerical propagator $\Phi^n$ defined in (\ref{phin def}) in the following lemma.

\begin{lemma}\label{lem:stab} (Stability)
Let $\gamma\ge 0$ and $v_0\in H^{\gamma+4}$, then there exist some constant $C,\tau_0>0$, such that for any $0<\tau\leq\tau_0$,
$$
\big\|\Phi^n\big(v(t_n)\big)-\Phi^n\big(v^n\big)\big\|_{H^\gamma}
\le \big(1+C\tau\big)\|v(t_n)-v^n\|_{H^\gamma},\quad n=0,1,\ldots,\frac{T}{\tau}-1,
$$
 where the constants $C,\tau_0$ depend only on $T$ and $\|v\|_{L^\infty((0,T);H^{\gamma+4})}$.
\end{lemma}
\begin{proof}
We denote $f_1=v(t_n)-v^n$, $f_2=\partial_tv(t_n)-{v}_t^{n}$, $g_1=v(t_n)+v^n$, $g_2=\partial_tv(t_n)+{v}_t^{n}$ for $0\leq n\leq T/\tau-1$, then
\begin{align*}
\Phi^n\big(v(t_n)\big)-\Phi^n\big(v^n\big)
=&f_1+\frac12\int_0^\tau \fe^{(t_n+s)\partial_x^3}\partial_x\Big(\fe^{-(t_n+s)\partial_x^3}f_1\cdot \fe^{-(t_n+s)\partial_x^3}g_1\Big)\,ds\\
&+\frac12\int_0^\tau s\fe^{(t_n+s)\partial_x^3}\partial_x\Big(\fe^{-(t_n+s)\partial_x^3}f_1\cdot \fe^{-(t_n+s)\partial_x^3}g_2\Big)\,ds\\
&+\frac12\int_0^\tau s\fe^{(t_n+s)\partial_x^3}\partial_x\Big(\fe^{-(t_n+s)\partial_x^3}f_2\cdot \fe^{-(t_n+s)\partial_x^3}g_1\Big)\,ds.
\end{align*}
Hence, we have
\begin{subequations}
\begin{align}
\quad&\big\|\Phi\big(v(t_n)\big)-\Phi\big(v^n\big)\big\|_{H^\gamma}^2\nonumber
\\\le& \|f_1\|_{H^\gamma}^2\notag\\
&+\Big\langle J^\gamma\int_0^\tau \fe^{(t_n+s)\partial_x^3}\partial_x\Big(\fe^{-(t_n+s)\partial_x^3}f_1\cdot \fe^{-(t_n+s)\partial_x^3}g_1\Big)\,ds,J^\gamma f_1\Big\rangle\label{0.17-1}\\
&+\Big\langle J^\gamma\int_0^\tau s\fe^{(t_n+s)\partial_x^3}\partial_x\Big(\fe^{-(t_n+s)\partial_x^3}f_1\cdot \fe^{-(t_n+s)\partial_x^3}g_2\Big)\,ds,J^\gamma f_1\Big\rangle\label{0.17-3}\\
&+\Big\langle J^\gamma\int_0^\tau s\fe^{(t_n+s)\partial_x^3}\partial_x\Big(\fe^{-(t_n+s)\partial_x^3}f_2\cdot \fe^{-(t_n+s)\partial_x^3}g_1\Big)\,ds,J^\gamma f_1\Big\rangle\label{0.17-2}\\
&+\frac34\Big\|\int_0^\tau \fe^{(t_n+s)\partial_x^3}\partial_x\Big(\fe^{-(t_n+s)\partial_x^3}f_1\cdot \fe^{-(t_n+s)\partial_x^3}g_1\Big)\,ds\Big\|_{H^\gamma}^2\label{0.17-4}\\
&+\frac34\Big\|\int_0^\tau s\fe^{(t_n+s)\partial_x^3}\partial_x\Big(\fe^{-(t_n+s)\partial_x^3}f_1\cdot \fe^{-(t_n+s)\partial_x^3}g_2\Big)\,ds\Big\|_{H^\gamma}^2\label{0.17-6}\\
&+\frac34\Big\|\int_0^\tau s\fe^{(t_n+s)\partial_x^3}\partial_x\Big(\fe^{-(t_n+s)\partial_x^3}f_2\cdot \fe^{-(t_n+s)\partial_x^3}g_1\Big)\,ds\Big\|_{H^\gamma}^2\label{0.17-5}.
\end{align}
\end{subequations}
Now we estimate \eqref{0.17-1}--\eqref{0.17-5} term by term.

We begin with estimate of \eqref{0.17-1}. 
Applying Lemma \ref{lm3.5}-(i), we get for any $\gamma\geq0$,
\begin{align*}
|\eqref{0.17-1}|
\lesssim  \tau\big\|f_1\big\|_{H^\gamma}^2 \big\|g_1\big\|_{H^{\gamma+2}}.
\end{align*}
From the \emph{a prior} estimate in Lemma \ref{lem:est-vn}, we have that when $0<\tau\leq \tau_0$,
\begin{align}
\big\|g_1\big\|_{H^{\gamma+2}}\le C, \label{est:g1}
\end{align}
for some $\tau_0,C$ depend on $T$ and $\|v\|_{L^\infty((0,T);H^{\gamma+4})}$.
Hence, we further obtain
\begin{align}
|\eqref{0.17-1}|
\le C\tau\big\|f_1\big\|_{H^\gamma}^2.\label{est:0.17-1}
\end{align}

Now we estimate the terms \eqref{0.17-3} and \eqref{0.17-6} which can be done in the same manner. To do this, by the formula
 $$
 g_2=\frac12\fe^{t_n\partial_x^3}\partial_x\Big[\big(\fe^{-t_n\partial_x^3}v(t_n)\big)^2+\big(\fe^{-t_n\partial_x^3}v^n\big)^2\Big],
 $$
and Lemma \ref{lem:tri-est}-(i), we have for $\gamma\ge0, \gamma_1>\frac12$,
\begin{align*}
\Big\|\int_0^\tau\!\! s &\fe^{(t_n+s)\partial_x^3}\partial_x\Big(\fe^{-(t_n+s)\partial_x^3}f_1\cdot \fe^{-(t_n+s)\partial_x^3}g_2\Big)\,ds\Big\|_{H^\gamma}
\lesssim
\tau\|f_1\|_{H^{\gamma+\gamma_1-1}}\Big(\big\|v(t_n)\big\|_{H^{\gamma+\gamma_1}}^2+\big\|v^n\big\|_{H^{\gamma+\gamma_1}}^2\Big).
\end{align*}
From Lemma \ref{lem:est-vn}, we further get
$$
\Big\|\int_0^\tau\!\! s \fe^{(t_n+s)\partial_x^3}\partial_x\Big(\fe^{-(t_n+s)\partial_x^3}f_1\cdot \fe^{-(t_n+s)\partial_x^3}g_2\Big)\,ds\Big\|_{H^\gamma}
\le C\tau\|f_1\|_{H^{\gamma}},
$$
where the constant $C>0$ depends only on $\|v\|_{L^\infty((0,T);H^{\gamma+4})}$.
By this estimate, we get
\begin{subequations}\label{est:0.17-3and6}
\begin{align}
|\eqref{0.17-3}|&\le C\tau\|f_1\|_{H^{\gamma}}^2,\label{est:0.17-3}\\
|\eqref{0.17-6}|&\le C \tau^2\|f_1\|_{H^\gamma}^2.\label{est:0.17-6}
\end{align}
\end{subequations}

Next, we treat the terms \eqref{0.17-2} and \eqref{0.17-5}  in the same manner. Using the relationship
 $$
 f_2=\frac12\fe^{t_n\partial_x^3}\partial_x\Big(\fe^{-t_n\partial_x^3}f_1\cdot\fe^{-t_n\partial_x^3}g_1\Big),
 $$
and Lemma \ref{lem:tri-est}-(i), we have for $\gamma\ge0, \gamma_1>\frac12$,
\begin{align*}
\Big\|\int_0^\tau s &\fe^{(t_n+s)\partial_x^3}\partial_x\Big(\fe^{-(t_n+s)\partial_x^3}f_2\cdot \fe^{-(t_n+s)\partial_x^3}g_1\Big)\,ds\Big\|_{H^\gamma}
\lesssim
\tau\|f_1\|_{H^\gamma}\|g_1\|_{H^{\gamma+\gamma_1}}^2.
\end{align*}
Using this estimate and \eqref{est:g1}, we get
\begin{subequations}\label{est:0.17-2and5}
\begin{align}
|\eqref{0.17-2}|&\le C\tau\|f_1\|_{H^\gamma}^2,\label{est:0.17-2}\\
|\eqref{0.17-5}|&\le C\tau^2\|f_1\|_{H^\gamma}^2.\label{est:0.17-5}
\end{align}
\end{subequations}

Now it is left to consider \eqref{0.17-4}. By Lemma \ref{lem:bi-est} and \eqref{est:g1}, we get for $\gamma\ge0, \gamma_1>\frac12$
\begin{align}
|\eqref{0.17-4}|\le C\tau\big\|f_1\big\|_{H^\gamma}^2\big\|g_1\big\|_{H^{\gamma+\gamma_1}}^2\le C\tau\big\|f_1\big\|_{H^\gamma}^2.\label{est:0.17-4}
\end{align}

Combining the estimates \eqref{est:0.17-1}-\eqref{est:0.17-4},  we conclude that
$$
\big\|\Phi^n\big(v(t_n)\big)-\Phi^n\big(v^n\big)\big\|_{H^\gamma}^2
\le \|f_1\|_{H^\gamma}^2+C\tau\big\|f_1\big\|_{H^\gamma}^2,
$$
where $C>0$ depends on $T$ and $\|v\|_{L^\infty((0,T);H^{\gamma+4})}$.  Since $\sqrt{1+C\tau}\sim 1+C\tau$ when $\tau$ is small enough, we finish the proof of the lemma.
\end{proof}

\subsection{Local error}\label{subsec2}
Next,  we have the following optimal estimate for the local error term $\mathcal{L}^n$ in \eqref{11.26}.
\begin{lemma}\label{lem:L} (Local error estimate) Let $\mathcal{L}^n$ be defined in \eqref{11.26} and $\gamma\ge 0$, then  we have
$$
  \|\mathcal{L}^n\|_{H^\gamma}\leq C\tau^3,\quad n=0,1\ldots,\frac{T}{\tau}-1,
$$
 where the constant $C>0$ depends only on $T$ and $\|u\|_{L^\infty((0,T);H^{\gamma+4})}$.
\end{lemma}
\begin{proof}
For simplicity, we denote for $n=0,1,\ldots,\frac{T}{\tau}-1$
$$
w^n(s)=v(t_n+s)-v(t_n)-s\partial_tv(t_n),\quad h^n(s)=v(t_n+s)+v(t_n)+s\partial_tv(t_n),\quad s\geq0.
$$
Then from the definition, we have
\begin{align}
\mathcal{L}^n=&\frac{1}{2}\int_0^\tau \fe^{(t_n+s)\partial_x^3}\partial_x\Big[\fe^{-(t_n+s)\partial_x^3}w^n(s)
\cdot\fe^{-(t_n+s)\partial_x^3}h^n(s)\Big]\,ds\label{local:eq1}\\
&+\frac{1}{2}\int_0^\tau
 s^2\fe^{(t_n+s)\partial_x^3}\partial_x\left[\fe^{-(t_n+s)\partial_x^3}
 \partial_t v(t_n)\right]^2ds.\nonumber
\end{align}
Noting that by Taylor's expansion,
$$
w^n(s)=\int_{0}^s\!\!\int_{0}^{s'}\partial_t^2v(t_n+t)\,dt ds',
$$
and then by the formula in Lemma \ref{lem:two-form}-(i), we see
$$
w^n(s)=\int_{0}^s\!\!\int_{0}^{s'}\fe^{(t_n+t)\partial_x^3}
\partial_{x}^2\left[\frac32\Big(\fe^{-(t_n+t)\partial_x^3}\partial_xv(t_n+t)\Big)^2+\frac13\Big(\fe^{-(t_n+t)\partial_x^3}v(t_n+t)\Big)^3\right]\,dt ds'.
$$
Plugging the above formula into (\ref{local:eq1}), we get
$
\mathcal{L}^n
 =\mathcal{L}^n_1+\mathcal{L}^n_2,
$
where
\begin{align*}
 \mathcal{L}^n_1:=&
 \frac12\int_0^\tau\!\!\int_{0}^s\!\!\int_{0}^{s'} \fe^{(t_n+s)\partial_x^3}
\partial_x\Bigg[\fe^{(t-s)\partial_x^3}\partial_{x}^2
\left[\frac32\left(\fe^{-(t_n+t)\partial_x^3}\partial_xv(t_n+t)\right)^2
+\frac13\left(\fe^{-(t_n+t)\partial_x^3}v(t_n+t)\right)^3\right]\\
&\cdot\fe^{-(t_n+s)\partial_x^3}h^n(s)\Bigg]\,dtds'ds,\\
\mathcal{L}^n_2:=&\frac{1}{2}\int_0^\tau
 s^2\fe^{(t_n+s)\partial_x^3}\partial_x\left[\fe^{-(t_n+s)\partial_x^3}
 \partial_t v(t_n)\right]^2ds.
\end{align*}
For $\mathcal{L}^n_1$, firstly we have
\begin{align*}
  \|\mathcal{L}^n_1\|_{H^\gamma}\lesssim&\int_0^\tau\!\!\int_{0}^s\!\!\int_{0}^{s'}\!\! \Bigg(\Big\|\fe^{(t_n+s)\partial_x^3}\partial_x\Big[\fe^{(t-s)\partial_x^3}
  \partial_{x}^2\big(\fe^{-(t_n+t)\partial_x^3}
  \partial_xv(t_n+t)\big)^2\cdot\fe^{-(t_n+s)\partial_x^3}h^n(s)\Big]
  \Big\|_{H^\gamma}\\
  &+\Big\|\fe^{(t_n+s)\partial_x^3}\partial_x\Big[\fe^{(t-s)\partial_x^3}
  \partial_{x}^2\Big(\fe^{-(t_n+t)\partial_x^3}v(t_n+t)\Big)^3\cdot\fe^{-(t_n+s)\partial_x^3}h^n(s)\Big]
  \Big\|_{H^\gamma}\Bigg)\,dtds'ds\\
  \lesssim&\int_0^\tau\!\!\int_0^\tau\!\!\int_0^\tau\!\! \Bigg(\Big\|\fe^{(t-s)\partial_x^3}\partial_{x}^2
  \big(\fe^{-(t_n+t)\partial_x^3}
  \partial_xv(t_n+t)\big)^2\cdot\fe^{-(t_n+s)\partial_x^3}h^n(s)
  \Big\|_{H^{\gamma+1}}\\
  &+\Big\|\fe^{(t-s)\partial_x^3}
  \partial_{x}^2\Big(\fe^{-(t_n+t)\partial_x^3}v(t_n+t)\Big)^3\cdot\fe^{-(t_n+s)\partial_x^3}h^n(s)
  \Big\|_{H^{\gamma+1}}\Bigg)\,dtds'ds.
\end{align*}
Then by using  Lemma \ref{lem:kato-Ponce} (i), we obtain
\begin{align*}
  \|\mathcal{L}^n_1\|_{H^\gamma}
  \lesssim&\int_0^\tau\!\!\int_0^\tau\!\!\int_0^\tau\!\! \bigg[\Big\|\partial_{x}^2\big(\fe^{-(t_n+t)\partial_x^3}
  \partial_xv(t_n+t)\big)^2\Big\|_{H^{\gamma+1}}\big\|h^n(s)\big\|_{H^{\gamma+1}}\\
  &+\Big\|\partial_{x}^2\big(\fe^{-(t_n+t)\partial_x^3}
  v(t_n+t)\big)^3\Big\|_{H^{\gamma+1}}\big\|h^n(s)\big\|_{H^{\gamma+1}}\bigg]\,dtds'ds\\
  \lesssim&\int_0^\tau\!\!\int_0^\tau\!\!\int_0^\tau\!\! \bigg[\Big\|\big(\fe^{-(t_n+t)\partial_x^3}
  \partial_xv(t_n+t)\big)^2\Big\|_{H^{\gamma+3}}\big\|h^n(s)\big\|_{H^{\gamma+1}}\\
  &+\Big\|\big(\fe^{-(t_n+t)\partial_x^3}
  v(t_n+t)\big)^3\Big\|_{H^{\gamma+3}}\big\|h^n(s)\big\|_{H^{\gamma+1}}\bigg]\,dtds'ds.
\end{align*}
Using Lemma \ref{lem:kato-Ponce} (i) again, we get that
\begin{align*}
  \|\mathcal{L}^n_1\|_{H^\gamma}
  \lesssim& \int_0^\tau\!\!\int_0^\tau\!\!\int_0^\tau\!\! \Big[\big\|v(t_n+t)\big\|_{H^{\gamma+4}}^2\big\|h^n(s)\big\|_{H^{\gamma+1}}
+ \big\|v(t_n+t)\big\|_{H^{\gamma+3}}^3\big\|h^n(s)\big\|_{H^{\gamma+1}}\Big]
  \,dtds'ds.
\end{align*}
Hence, in sum, we get
\begin{align}
  \|\mathcal{L}^n_1\|_{H^\gamma}
  \le& C\int_0^\tau\!\!\int_0^\tau\!\!\int_0^\tau\!\! \Big(\big\|v(t_n+t)\big\|_{H^{\gamma+4}}^2+\big\|v(t_n+t)\big\|_{H^{\gamma+4}}^3\Big)
  \big\|h^n(s)\big\|_{H^{\gamma+1}}\,dtds'ds.\label{23.29}
\end{align}
Now we control the term $h^n(s)$. From \eqref{vt eq} and  the Kato-Ponce inequality in Lemma \ref{lem:kato-Ponce}, we have
\begin{align}
\big\|h^n(s)\big\|_{H^{\gamma+1}}
\lesssim & \|v(t_n+s)\|_{H^{\gamma+1}}+\|v(t_n)\|_{H^{\gamma+1}}
+s\|\partial_tv(t_n)\|_{H^{\gamma+1}}\notag\\
\lesssim & \|v(t_n+s)\|_{H^{\gamma+1}}+\|v(t_n)\|_{H^{\gamma+1}}
+s\Big\|\left(\fe^{-t_n\partial_x^3}v(t_n)\right)^2\Big\|_{H^{\gamma+2}}\notag\\
\lesssim & \|v(t_n+s)\|_{H^{\gamma+1}}
+\|v(t_n)\|_{H^{\gamma+1}}
+s\big\|v(t_n)\big\|_{H^{\gamma+2}}^2\notag\\
\lesssim & \|v\|_{L^\infty((0,T);H^{\gamma+2})}
+\tau\|v\|_{L^\infty((0,T);H^{\gamma+2})}^2.\label{est:h}
\end{align}
Inserting this estimate into \eqref{23.29}, we get
$$
 \|\mathcal{L}^n_1\|_{H^\gamma}\leq C\tau^3,
$$
where $C$ depends on $\|v\|_{L^\infty((0,T);H^{\gamma+4})}$.

For $\mathcal{L}^n_2$, similarly as above, we have
\begin{align*}
  \|\mathcal{L}^n_2\|_{H^\gamma}\lesssim & \int_0^\tau s^2\Big\|\partial_x\left[\fe^{-(t_n+s)\partial_x^3}
 \partial_t v(t_n)\right]^2ds\Big\|_{H^\gamma}\,ds\\
  \lesssim & \int_0^\tau s^2\Big(\big\|\partial_x\partial_tv(t_n)\big\|_{H^\gamma}\big\| \partial_tv(t_n)\big\|_{H^{\gamma_1}}+\big\|\partial_x\partial_tv(t_n)\big\|_{L^2}\big\| \partial_tv(t_n)\big\|_{H^{\gamma+\gamma_1}}\Big)\,ds.
\end{align*}
Similarly as \eqref{est:h}, we obtain that
\begin{align*}
  \|\mathcal{L}^n_2\|_{H^\gamma}\lesssim &\tau^3\|v\|_{L^\infty((0,T);H^{\gamma+2})}.
\end{align*}
Combining the estimates on $\mathcal{L}^n_1$ and $\mathcal{L}^n_2$, we finish the proof of the lemma.
\end{proof}

\subsection{Proof of Theorem \ref{thm:convergence}} \label{subsec4}
Now, combining the local error estimate and the stability results, we give the proof of Theorem \ref{thm:convergence}.
As described in the subsection \ref{sub1p5}, it is sufficient to estimate $\|v(t_n)-v^n\|_{H^\gamma}$. From \eqref{11.22}, Lemma \ref{lem:L} and Lemma \ref{lem:stab}, there exit constants $C>0$ and $\tau_0>0$ (from Lemma \ref{lem:stab}), such that for $0<\tau\leq\tau_0$, we have
\begin{align*}
\|v(t_{n+1})-v^{n+1}\|_{H^\gamma}
\le C\tau^3+(1+C\tau)\|v(t_{n})-v^{n}\|_{H^\gamma},\quad n=0,1,\ldots,\frac{T}{\tau}-1,
\end{align*}
where $C,\tau_0$ depend on $T$ and $\|v\|_{L^\infty((0,T);H^{\gamma+4})}$.
By iteration and Gronwall's inequality, we get
\begin{align*}
\|v(t_{n+1})-v^{n+1}\|_{H^\gamma}
\le \tau^3\sum\limits_{j=0}^n(1+C\tau)^j\le C \tau^2,\quad n=0,1,\ldots,\frac{T}{\tau}-1,
\end{align*}
which proves Theorem \ref{thm:convergence}.\qed

\section{Numerical results}\label{sec:numerical}
In this section, we carry out numerical experiments of the presented LRI scheme (\ref{LRI}) for justifying the convergence theorem. Also, we provide the numerical investigations of convergence of the Strang splitting scheme \cite{splitting1,splitting2} (or see the Appendix \ref{appendix}) as comparisons.

To get an initial data with the desired regularity, we construct $u_0(x)$ by the following strategy \cite{lownls}.
 Choose $N>0$ as an even integer and discrete the spatial domain $\bT$ with grid points
$x_j=j\frac{2\pi}{N}$ for $j=0,\ldots,N$.
Take a uniformly distributed random vectors $\mathrm{rand}(N,1)\in[0,1]^N$ and denote
$$\mathcal{U}^N=\mathrm{rand}(N,1).$$
Then we define
\begin{equation}\label{non-smooth}
u_0(x):=\frac{|\partial_{x,N}|^{-\theta}\mathcal{U}^N}
{\||\partial_{x,N}|^{-\theta}\mathcal{U}^N\|_{L^\infty}},\quad
x\in\bT,
\end{equation}
where the pseudo-differential operator $|\partial_{x,N}|^{-\theta}$ for $\theta\geq0$ reads: for Fourier modes $l=-N/2,\ldots$, $N/2-1$,
\begin{equation*}
 \left(|\partial_{x,N}|^{-\theta}\right)
 _l=\left\{\begin{split}
 &|l|^{-\theta}\quad \mbox{if}\ l\neq0,\\
  &0\qquad\ \, \mbox{if}\ l=0.
  \end{split}\right.
\end{equation*}
Thus, we get $u_0\in H^\theta(\bT)$ for any $\theta\geq0$.  We implement the spatial discretizations of the numerical methods within discussions by the Fourier pseudo-spectral method \cite{Shen} with a large number of grid points $N=2^{12}$ in the torus domain $\bT$. We shall present the error
$u(x,t_n)-u^n$ in the $H^{\gamma}$-norm ($\gamma=0$ or $2$) at the final time $t_n=T=2$, where the exact solution is obtained numerically by the LRI scheme (\ref{LRI}) with $\tau=10^{-4}$. Figure \ref{fig:LRI} shows the convergence results of the LRI scheme (\ref{LRI}) by using different time step $\tau$ under the initial data of different regularities. In Figure \ref{fig:strang}, we show the corresponding convergence curves of the Strang splitting scheme (\ref{Strang}) from \cite{splitting1,splitting2}. The details of the implementations of the Strang splitting scheme is given in the Appendix \ref{appendix}.

\begin{figure}[t!]
$$\begin{array}{cc}
\psfig{figure=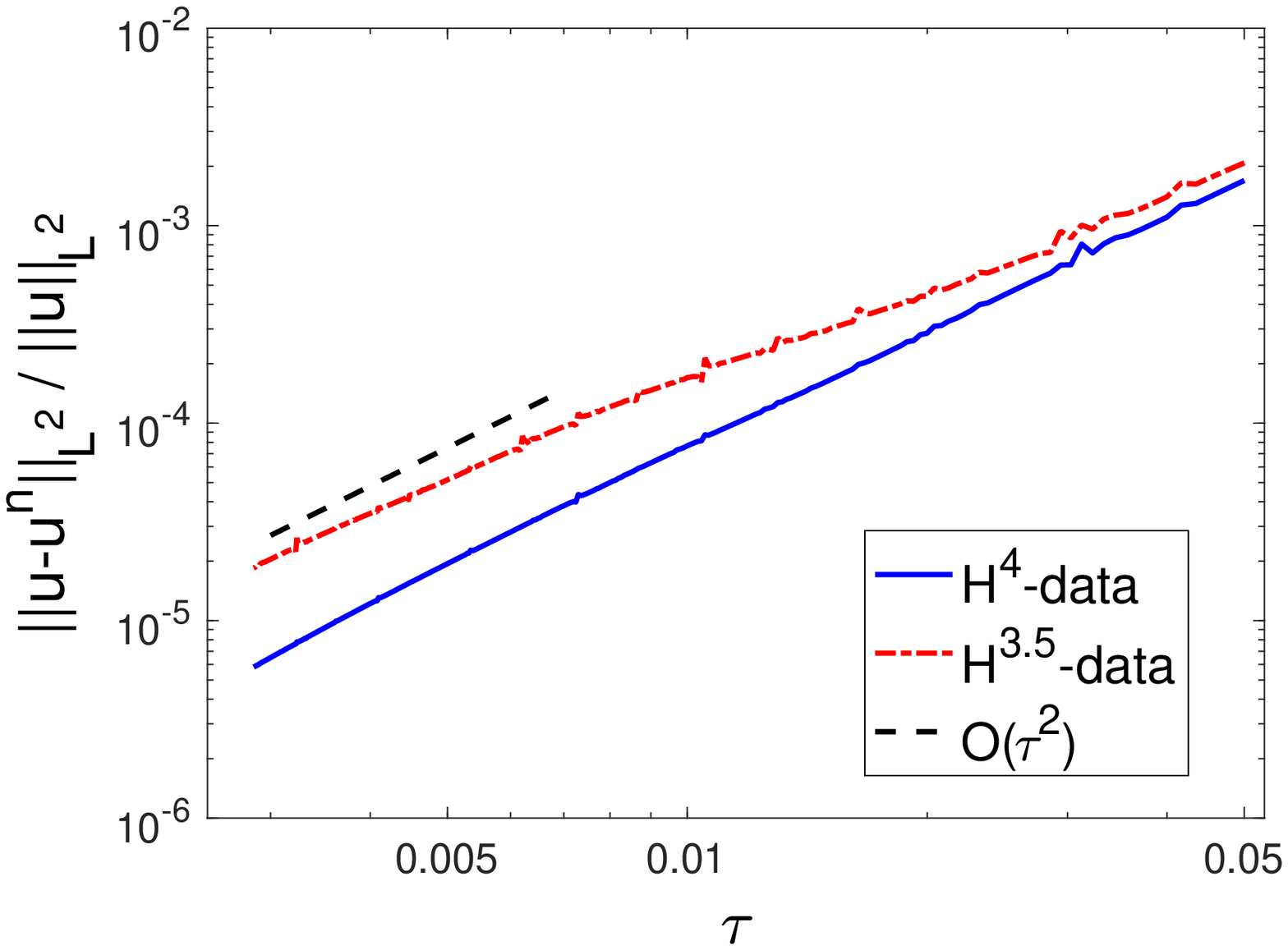,height=6cm,width=7cm}&
\psfig{figure=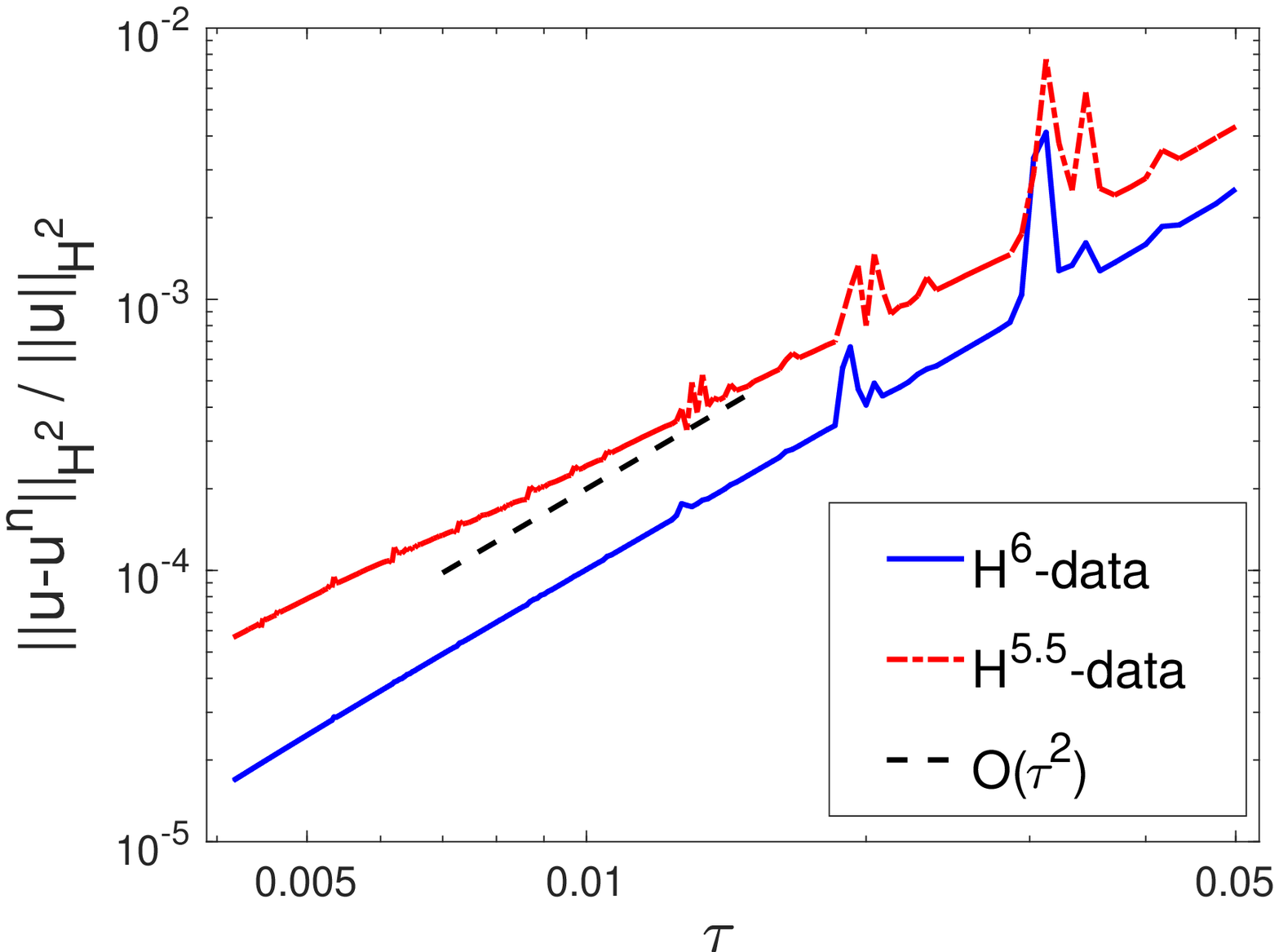,height=6cm,width=7cm}
\end{array}$$
\caption{Convergence of the LRI scheme: relative error $\|u-u^n\|_{L^2}/\|u\|_{L^2}$ (left) and $\|u-u^n\|_{H^2}/\|u\|_{H^2}$ (right) at $t_n=T=2$ under initial data of different regularities.
}
\label{fig:LRI}
\end{figure}

\begin{figure}[t!]
$$\begin{array}{cc}
\psfig{figure=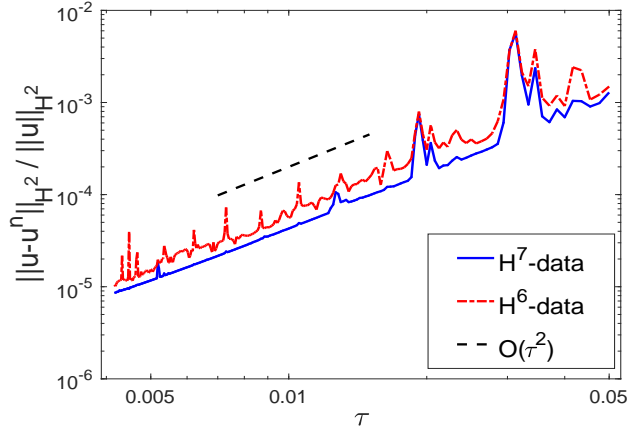,height=6cm,width=9cm}
\end{array}$$
\caption{Convergence of the Strang splitting scheme: relative error $\|u-u^n\|_{H^2}/\|u\|_{H^2}$ at $t_n=T=2$ under initial data of different regularities.}
\label{fig:strang}
\end{figure}

Based on the numerical results from Figures \ref{fig:LRI} \& \ref{fig:strang}, we have the following observations:

1) The presented LRI scheme (\ref{LRI}) has the second order accuracy in time under $H^{\gamma}$-norm with initial data in $H^{\gamma+4}$ for any $\gamma\geq0$ (see the blue solid lines in Figure \ref{fig:LRI}), while with less regularity than $H^{\gamma+4}$ (see the red dash-dot lines in Figure \ref{fig:LRI}), the LRI scheme shows some convergence order reduction. This indicates that our theoretical estimate in Theorem \ref{thm:convergence} is optimal and the regularity assumption is sharp.

2) The Strang splitting scheme (\ref{Strang}) converges at the second order rate in $H^\gamma$ with initial data in $H^{\gamma+5}$ (see the blue solid line in Figure \ref{fig:strang}), which confirms the theoretical result proved in \cite{splitting2}.  With less regular initial data, e.g. $H^{\gamma+4}$ initial data, the scheme  still converges but with an unstable order (see the red dash-dot line in Figure \ref{fig:strang}). The implicity of Strang splitting scheme makes  the computations very time-consuming.

3) The error from LRI (\ref{LRI}) and the Strang splitting scheme (\ref{Strang}) are rather similar (cf. the right one in  Figure \ref{fig:LRI} and  Figure \ref{fig:strang}), while the LRI (\ref{LRI}) is much more efficient.

\section{Conclusion} \label{sec:conclusion}
In this work, we have studied numerically the KdV equation on a torus under rough initial data. By some rigorous tools from harmonic analysis, we established the sharp convergence theorem of an exponential-type integrator as outlined in  \cite{kdv-kath}. The theoretical result shows that the presented integrator can reach the second order accuracy in $H^\gamma$ space with initial data from $H^{\gamma+4}$ for any $\gamma\geq0$. Compared with classical numerical methods, the presented integrator requires less regularity of the solution for optimal convergence rate and is more efficient for solving the KdV equation under rough initial data case.

\appendix

\section{Strang splitting scheme}\label{appendix}
As firstly used in \cite{splitting0}, the Strang splitting method applies to the KdV equation (\ref{model}) by splitting it into a linear part:
\begin{equation*}
  \Phi_A^t:\quad \partial_t u(t,x)+\partial_x^3u(t,x)=0,\quad t>0,\ x\in\bT,
\end{equation*}
and an inviscid Burgers equation:
\begin{equation*}
  \Phi_B^t:\quad \partial_t u(t,x)=\frac{1}{2}\partial_x(u(t,x))^2,\quad t>0,\ x\in\bT,
\end{equation*}
where $\Phi_A^t(\cdot)$ and $\Phi_B^t(\cdot)$ denote the propagators. Then the Strang splitting scheme reads: denote $u^n=u^n(x)\approx u(t_n,x)$ and for $n\geq0$,
\begin{equation}\label{Strang}
  u^{n+1}=\Phi_A^{\tau/2}\circ\Phi_B^{\tau}\circ\Phi_A^{\tau/2}(u^n).
\end{equation}
The propagator $\Phi_A^t(u)=\fe^{-t\partial_x^3}u$ is given exactly. Here to implement the Strang splittng scheme as has been analyzed in \cite{splitting1,splitting2}, we seek for the exact solution at the Burgers step (at least up to machine precision). The solution of $\Phi_B^t(u)$ can be given by the characteristics method as follows. For $x_0\in\bT$, let $x=x(t)$ satisfying
$$\dot{x}(t)=-u(t,x(t)),\quad t>0,\quad x(0)=x_0.$$
Along the characteristics we have
$\frac{d}{dt}u(t,x(t))=0$, and so
$
 \dot{x}(t)=-u(0,x_0)
$ which gives
$$x(t)-x_0=-tu(0,x_0),\quad t\geq0.$$
Hence, with $u(0,x)=u_0(x)$ known, if we want to compute $u(t,x_j)$ at the grid point $x_j\in\bT$, we set $x(t)=x_j$ and so $u(t,x_j)=u_0(x_0)$. Then  we  solve the nonlinear equation $x_j=x_0-tu_0(x_0)$ for the initial position $x_0$, which can be done by for example the Newton's iteration. Afterwards, we interpolate $u_0$ at $x_0$,
which can be obtained accurately by the non-uniform fast Fourier transform (NUFFT) \cite{NUFFT}. In our implementation, we apply the NUFFT to the accuracy $\delta=10^{-13}$ and the same for the Newton's iteration: $\delta=x_j-x_0^{(n)}+tu_0\left(x_0^{(n)}\right)$. The full scheme is implicit.

\section*{Acknowledgements}
Y. Wu is partially supported by NSFC 11771325 and 11571118.  X. Zhao is partially supported by the Natural Science Foundation of Hubei Province No. 2019CFA007, the NSFC 11901440 and the starting research grant of Wuhan University. Part of the work was done while the authors were visiting the Shanghai Center for Mathematical Sciences.


\bibliographystyle{model1-num-names}

\end{document}